\documentclass[preprint]{imsart}
\RequirePackage[OT1]{fontenc}
\RequirePackage{amsthm,amsmath}
\RequirePackage[numbers]{natbib}
\RequirePackage[colorlinks,citecolor=blue,urlcolor=blue]{hyperref}

\usepackage{graphicx} % more modern
\usepackage{algorithm}
\usepackage{algorithmic}
\usepackage{hyperref}
\usepackage{amsmath}%
\usepackage{amsfonts}%
\usepackage{amssymb}%
\usepackage{amsthm}% Enhanced theorem environment
\usepackage{mathtools}
\usepackage{color}
\usepackage{capt-of}% or \usepackage{caption}
\usepackage[update,prepend]{epstopdf}

\DeclarePairedDelimiter\abs{\lvert}{\rvert}%
\DeclarePairedDelimiter\norm{\lVert}{\rVert}%

\makeatletter
\let\oldabs\abs
\def\abs{\@ifstar{\oldabs}{\oldabs*}}
\let\oldnorm\norm
\def\norm{\@ifstar{\oldnorm}{\oldnorm*}}
\makeatother

\newtheorem{theorem}{Theorem}[section]
\theoremstyle{plain}

\newtheorem{corollary}[theorem]{Corollary}

\newtheorem{definition}[theorem]{Definition}

\newtheorem{lemma}[theorem]{Lemma}

\newcommand{\X}{{\mathcal X}}
\newcommand{\Y}{{\mathcal Y}}
\newcommand{\bqn}{\begin{eqnarray}}
\newcommand{\eqn}{\end{eqnarray}}
\newcommand{\bqns}{\begin{eqnarray*}}
\newcommand{\eqns}{\end{eqnarray*}}
\newcommand{\ny}{\left\{}
\newcommand{\zr}{\right\}}

\newcommand{\argmin}[1]{\underset{#1}{\text{arg min}}}

\newcommand{\cov}{\text{Cov}}

\newcommand{\E}{\text{E}}
\newcommand{\Eins}{\mathbf{1}}
\newcommand{\half}{\frac{1}{2}}
\newcommand{\iid}{\overset{\text{\tiny i.i.d}}{\sim}}

\newcommand{\lmax}{\lambda_{\max}}

\newcommand{\pr}{\text{Pr}}

\newcommand{\R}{\mathbb{R}}

\newcommand{\var}{\text{Var}}

\renewcommand{\(}{\left (}
\renewcommand{\)}{\right )}
\renewcommand{\geq}{\geqslant}
\renewcommand{\H}{\mathcal{H}}

\renewcommand{\leq}{\leqslant}

\newcommand{\N}{\mathcal{N}}
\renewcommand{\H}{\mathcal{H}}

\renewcommand{\(}{\left (}
\renewcommand{\)}{\right )}

\renewcommand{\d}{\text{d}}

\newcommand{\T}{\mathcal{T}}

\newcommand{\hhr}{\hat{h}_{R}}
\newcommand{\sign}{\text{sign}\,}
\newcommand{\x}{h}
\newcommand{\y}{x}
\newcommand{\loss}{\ell}%{\mathcal{L}}
\newcommand{\erm}{\hat{h}}
\newcommand{\hypclass}{\mathcal{H}}
\newcommand{\train}{\mathcal{T}^N}
\newcommand{\rad}{\mathcal{R}}
\newcommand{\ddd}{\mathcal{D}}
\newcommand{\sgn}{\text{sgn}}

\newcommand{\phix}{x} % \phix replaced

\arxiv{arXiv:0000.0000}

\startlocaldefs
\numberwithin{equation}{section}
\theoremstyle{plain}

\endlocaldefs

\begin{document}

\begin{frontmatter}
\title%{Compressive error bounds for linear classifiers with the 0-1 loss}%\thanksref{T1}}
{Structure-aware error bounds for linear classification with the zero-one loss}%\thanksref{T1}}
\runtitle%{Compressive error bounds with the 0-1 loss}
{Structure-aware error bounds with the zero-one loss}
%\thankstext{T1}{Footnote to the title with the ``thankstext'' command.}

\begin{aug}
%\author{\fnms{Ata} \snm{Kab\'an}\thanksref{m1}\ead[label=e1]{A.Kaban@cs.bham.ac.uk}}%\thanksref{t1,t2,m1}
%\and
%\author{\fnms{Robert J.} \snm{Durrant}\thanksref{m2}\ead[label=e2]{BobD@waikato.ac.nz}}%\thanksref{t3,m1,m2}
\author{\fnms{Ata} \snm{Kab\'an}\ead[label=e1]{A.Kaban@cs.bham.ac.uk}}%\thanksref{t1,t2,m1}
\and
\author{\fnms{Robert J.} \snm{Durrant}\ead[label=e2]{BobD@waikato.ac.nz}}%\thanksref{t3,m1,m2}
%\and
%\author{\fnms{Third} \snm{Author}\thanksref{t1,m2}
%\ead[label=e3]{email@foo.com}
%\ead[label=u1,url]{http://www.foo.com}}

\if 0
\thankstext{t1}{Some comment}
\thankstext{t2}{First supporter of the project}
\thankstext{t3}{Second supporter of the project}
\fi

\runauthor{A.Kab\'an and R.J. Durrant}

\affiliation{University of Birmingham\thanksmark{m1} and University of Waikato\thanksmark{m2}}

\address{%Address of the First author\\
School of Computer Science\\
The University of Birmingham\\
Edgbaston\\
Birmingham B15 2TT\\ 
UK\\
\printead{e1}\\
%\phantom{E-mail:\ }\printead*{e2}
}

\address{%Address of the Second author\\
Department of Statistics\\
University of Waikato\\
Private Bag 3105\\
Hamilton 3240\\
New Zealand\\
\printead{e2}\\
%\printead{u1}
}
\end{aug}

\begin{abstract}
We prove risk bounds for binary classification in high-dimensional settings when the sample size is allowed to be smaller than the dimensionality of the training set observations.
In particular, we prove upper bounds for both `compressive learning' by empirical risk minimization (ERM) -- that is when the ERM classifier is learned from data that have been projected from high-dimensions onto a randomly selected low-dimensional subspace -- as well as uniform upper bounds in the full high-dimensional space.
A novel tool we employ in both settings is the `flipping probability' of Durrant and Kab\'{a}n (ICML 2013) which we use to %expose several 
capture benign geometric structures that make a classification problem `easy' in the sense of demanding a relatively low sample size for guarantees of good generalization. Furthermore our bounds also enable us to explain or draw connections between several existing successful classification algorithms. Finally we show empirically that our bounds are informative enough in practice to serve as the objective function for learning a classifier (by using them to do so).
\end{abstract}



\end{frontmatter}

%%%%%%%%%%%%%%%%%%%%%%%%%%%%%%%%%%%%%%%%%%%%%%%%%%%%%%%
\section{Introduction}% Do we need a soft intro or can go directly with the following?
% Can do either. An unscientific small survey of recent accepted papers suggests non-technical preamble - if any - should be brief.

%\subsection{Preliminaries and notations}
Given a function class $\hypclass$ and a training set of $N$ observations $\train=\{( x_n,y_n) :  
(x_{n},y_{n}) \iid \ddd \}_{n=1}^{N}$, where $\ddd$ is an unknown distribution over $\X\times \Y, \X\subseteq \R^d, \Y=\{-1,1\}$, 
the goal in binary classification is to use the training data to find a function (classifier) $\erm\in \hypclass$ so that, with respect to some given loss function $\ell$, its generalization error (or risk):
%$\pr_{(x,y) \sim \ddd}\left \{h(x) \ne y\right \} = 
\bqn
\E[\loss\circ\erm]:= \E_{(x,y) \sim \ddd}[\ell(\erm(x),y) | \mathcal{T}^N]
\eqn
is as small as possible. For binary classification  the function $\ell: \Y\times \Y \rightarrow \{0,1\}, \ell(\erm(x),y) = \Eins(\erm(x) \neq y)$, called the zero-one loss, is the main error measure of interest \cite{01}. Here $\Eins(\cdot)$ denotes the indicator function which returns one if its argument is true and zero otherwise. 
The optimal classifier in $\hypclass$ is denoted by $h^*:=\argmin{h\in \hypclass}~\E_{(x,y)\sim{\mathcal D}}[\loss(h(x),y)]$. 

In this work we consider functions of the form $\hypclass :=   
\{x\rightarrow \sign(h^Tx) : h\in \R^d, x\in \X \}$,
that is $\hypclass$ is identified with the set of normals to hyperplanes which, without loss of generality (since otherwise we can always concatenate a 1 to all inputs and work in $\R^{d+1}$ instead of $\R^d$), pass through the origin. 
 An extension to a convex combination of binary classifiers will also be considered in Section \ref{sec:boost}.

%---

Since $\ddd$ is unknown, one cannot minimize the generalization error directly. Instead, we have access to the empirical error over the training set -- the minimizer of which is the Empirical Risk Minimizer or ERM classifier:
\[
\erm :=\argmin{h\in \hypclass}~\frac{1}{N}\sum_{n=1}^N \loss(h(x_i),y_i).
\]

It is well known that for linear classification with the zero-one loss, in the absence of any further assumptions, the difference between the generalization error and the empirical error of a linear function class is  $\tilde{\Theta}(\sqrt{d/N})$ \cite{Massart}, so for meaningful guarantees -- when we are agnostic about the properties of the data generator -- we need the sample size $N$ be of order $d$.

However in this work we are interested in the case when $d$ is large compared to $N$. Often in such cases a  dimension-reducing preprocessing is used in practice -- for example to ensure model identifiability or as a form of regularization. The most common methods for dimensionality reduction in practice are linear transformations such as Principal Components Analysis, Factor Analysis, Independent Components Analysis, and so on. Here we consider Random Projection (RP) which is a computationally cheap, yet theoretically well-motivated, random linear dimensionality technique that is also supported by novel data acquisition devices developed in the area of compressed sensing. Moreover RP is oblivious to the data and amenable to analysis.\\

We will discuss RP in more detail in Section  \ref{sec:tools}, meanwhile let $R \in \R^{k \times d}$, $k \le d$ denote an instance of such a random mapping. A convenient way to think about $R$ is as a means of compressing or `sketching' the training data, and to compress we simply sample a single random matrix $R$ of a particular kind and then left multiply the observations with it.\\
Now let $\mathcal{T}^{N}_{R}=\{( Rx_{n},y_{n})\}_{n=1}^{N}$ denote the RP of the training set, so the input points $Rx_n$ are now $k$-dimensional. The hypothesis class defined on such $k$-dimensional inputs will be denoted by  $\hypclass_R:=\{Rx\rightarrow \sign(h_R^T Rx +b) : h_R\in \R^k, b \in \R, x\in\X \}$. Other analogous notations will be used in the $k$-dimensional space: 
$h^*_R=\argmin{h_R\in \hypclass_R}~\E[\loss \circ h_R]$ denotes the optimal classifier in $\hypclass_R$,
and
$\erm_R =\argmin{h_R\in \hypclass_R}~\frac{1}{N}\sum_{n=1}^N \loss(h_R(Rx_i),y_i)$ the ERM in $\hypclass_R$.
For any particular instance of $R$ the learned ERM classifier in the corresponding $k$-dimensional subspace $\erm_{R}$ is possibly not through the origin, but any non-zero translation $b$ will not affect our proof technique.

The generalization error of $\erm_R$, which is a random variable depending on both $\train$ and $R$, is then:
\begin{equation}
\E[\loss\circ\erm_R]:=\E_{(x,y)\sim \mathcal D}\left[\loss(\erm_{R}(Rx),y) | \mathcal{T}^{N}, R\right] 
\end{equation}

The remainder of this paper is motivated by two goals: 
\begin{itemize}
\item Firstly, we are interested in generalization guarantees for compressive ERM classification, and in finding out what structural characteristics of the data generator they depend on. In other words in terms of classification error, if we work with an RP sketch of the data how much, if anything, does it cost us? This is pursued in Section \ref{sec:RP}.
\item Secondly, we are interested in exploiting the insights gained in Section \ref{sec:RP} with respect to the original high dimensional problem, and to derive from them a better understanding of learning and generalization that takes account of the compressibility of problems in order to improve uniform guarantees on generalization. %In other words, given data with a particular structure, what is the minimum sample size required to guarantee good generalization with high probability? % we didn't do that.
%In other words, given data with some unknown structure, what generalization can we guarantee as a function of the structure? 
This is the subject of Section \ref{sec:Dataspace}.
\end{itemize}

Section \ref{sec:tools} introduces the main technical tools and concepts that we will be using to pursue these goals, while proofs are provided in Section \ref{sec:Pfs}. 

%%%%%%%%%%%%%%%%%%%%%%%%%%%%%%%%%%%%%%%%%%%%%%%%%%%%%%%%%%%%%%
\subsection{Tools and definitions}
\label{sec:tools}

\subsubsection{Random Projection}
Random projection takes advantage of a blessing of dimensionality, namely concentration of measure, as a means by which high dimensional data may be `sketched' with quite high fidelity in a much lower dimensional space. The prototypical result in this direction is the Johnson-Lindenstrauss Lemma (JLL) \cite{jll84}, %the usual style of proof for which 
a constructive proof of which
can be found in e.g. \cite{DasGup02} and shows that if we sample a $k \times d$ matrix $R$ with independent zero-mean Gaussian rows, $k \leq d$, then -- up to a small multiplicative error -- Euclidean norms and dot products can be preserved under left multiplication by $R$ with high probability. Similar results can be obtained for matrices with subgaussian rows and for Haar-distributed matrices \cite{jll84, Achlioptas, Matousek}. The JLL has motivated approximation algorithms for a wide variety of applications, including for classification \cite{Vempala,Cannings}.

Along very different lines to the JLL the probability that a dot product changes its sign as a result of RP, which we call the \emph{flipping probability}, was first introduced in \cite{icml13} and is the main novel theoretical tool used in developing our guarantees in the later sections. In particular we use this as a way to measure the difficulty or easiness of a particular classification problem in terms of its compressibility.\\
We also employ classical measures of complexity.  %  that are reviewed in the Appendix? Q
In particular, we employ the Gaussian width (or Gaussian complexity) as a natural, yet novel means to capture the geometry of the input domain that uniformly guarantees no sign flipping under RP with high probability, and allows us to derive sufficient conditions for `structure-aware' guarantees. We employ the classical VC-dimension and Rademacher complexity to measure the complexity of function classes -- however due to our dimensionality reduction approach the latter %the complexity of the function class 
will only be required for function classes defined on the low dimensional compressed input domain.
%We build these results into classical uniform bounds employing VC-dimension or Rademacher complexities for function classes on a compressed input space to obtain tighter guarantees.

\subsubsection{Flipping probability}\label{flips}
\begin{lemma}[Flipping probability, Gaussian RP matrix]
\label{thm:flip}
Let $R$ be a RP matrix with entries $r_{ij} \iid \mathcal{N}(0,\sigma^2)$,
let $\x,\y \in \R^{d}$, and let $\theta=\theta_{\y}^{\x} \in [0,\pi)$ be the angle between them.
Let $R\x,R\y  \in \R^{k}$ be the images of $\x,\y$ under $R$. Then the following hold:
\begin{enumerate}	
	\item   Flipping probability exact form. For $h^Tx\neq 0$, 
		\begin{eqnarray*}
		f_k(\theta) := 
		\frac{\Gamma(k)}{(\Gamma(k/2))^{2}}\int_{0}^{\frac{1-\cos(\theta)}{1+\cos(\theta)}}
	   	    \frac{z^{(k-2)/2}}{(1+z)^{k}} \d z	\label{eq:zform}
	   	    =\pr\ny (R\x)^TR\y\le 0 \zr
	   	    \\
	    \pr\ny \frac{(R\x )^{T}R\y}{\x^T\y} \le 0\zr  = f_k(\theta)\cdot\Eins(\x^T\y > 0) +
	           (1-f_k(\theta))\cdot\Eins(\x^T\y < 0)
		\end{eqnarray*}
%	where $\psi = (1-\cos(\theta))/(1+\cos(\theta))$.
	
	\item	Flipping probability upper bound. If $\x^T\y\neq 0$, then: %\cos(\theta)>0$, 
	%we have the following upper-bound:
		\begin{equation}
		\pr\ny \frac{(R\x )^{T}R\y}{\x^T\y}  \le 0 \zr \le \exp(-k\cos^2(\theta)/2) \label{ball}
		\end{equation}	
\end{enumerate}
\end{lemma}
	
%For acute angles, $f_k(\cdot)$ is the probability that the sign of dot product flips from positive to negative after RP, and for obtuse angles it is 1 minus the probability that it flips from negative to positive. 
The flipping probability of a Gaussian random matrix is a good approximation of that of the Haar-distributed matrix (i.e. uniformly randomly oriented $k \times d$ matrices with orthonormal rows), and has the advantage that it can be computed exactly. Figure \ref{fig:flipprob} provides a visualization of $f_{k}(\theta)$ -- note that the flipping probability depends only on the angle between a pair of vectors and the projection dimension $k$, it is independent of the dimensionality of the data, $d$. Furthermore, unlike the theory developed for randomly-projected ensemble classifiers in \cite{Cannings}, no sufficient dimensionality reduction conditions are required in our theory for this independence on the original data dimension to hold.

\begin{figure}
\centering
\includegraphics[width=9cm,height=7cm]{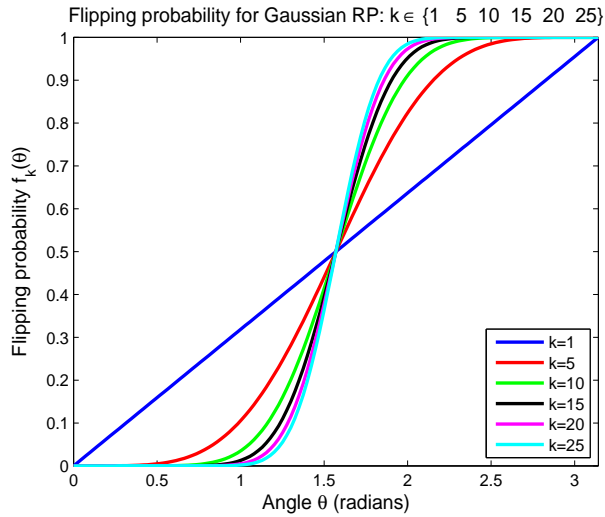}
\caption{\label{fig:flipprob}\vspace{0.2cm} Illustration of the function $f_k$ from the Flipping Probability Lemma \ref{thm:flip}.}
\end{figure}

Lemma \ref{thm:flip} appears in \cite{icml13}, and we include a simplified proof in Section \ref{sec:Pfs:flips}.
Below we also extend the applicability of this concept to subgaussian RPs, which may enjoy better computational efficiency \cite{Achlioptas, Matousek}. The following definitions can be found in e.g.  \cite{BulKov,Vbook}.

\begin{definition}[Subgaussian random variable] 
A zero-mean random variable $X$ is subgaussian with parameter $\sigma^2$ if $\exists \sigma^2>0$ such that:
\begin{eqnarray}
\E\ny \exp(\lambda X) \zr \le \exp\ny \sigma^2\lambda^2/2\zr \label{def:sg}
\end{eqnarray}
In other words, the moment generating function of $X$ is dominated by the moment generating function of a zero-mean Gaussian that has variance $\sigma^2$.\\
It will also be useful to introduce here the subgaussian norm of $X$, which is defined as:
\begin{equation} 
\|X\|_{\psi_2} = \sup_{p \geq 1}\left\{\frac{\E [|X|^p]^{1/p}}{\sqrt{p}}\right\}
\end{equation}
Thus the class of subgaussian random variables over a given probability space comprises a normed space. Moreover one has under the conditions of Definition \ref{def:sg}:
\begin{equation} 
\E\ny \exp(\lambda X) \zr \le \exp\ny C\lambda^2 \|X\|_{\psi_2}^2\zr
\end{equation}
for some strictly positive constant $C$ and we see that the subgaussian norm is proportional to the standard deviation of $X$, namely $\|X\|_{\psi_2} \propto \sigma$.
\end{definition}
%For example, it is easy to see (by Taylor expansion) that a 0-mean subgaussian  random variable has variance upper bounded by $\sigma^2$.   
\begin{lemma}[Flipping probability upper bound, subgaussian case] \label{thm:flip2}
Let $R$ be a RP matrix with entries $r_{ij}$ drawn i.i.d. from a zero-mean subgaussian distribution, % with parameter $\sigma^2$.
let $\x,\y \in \R^{d}$, and let $\theta=\theta_{\y}^{\x}$ be the angle between them. Let $R\x,R\y  \in \R^{k}$ be the images of $x,y$ under $R$. 
Then, if $\x^T\y\neq 0$, we have: %$\cos(\theta)>0$, we have the following upper-bound:
		\begin{equation}
		\pr\ny \frac{(R\x )^{T}R\y}{\x^T\y}  \le 0\zr \le \exp(-k\cos^2(\theta)/8) 
		\end{equation}	
\end{lemma}

%%%%%%%%%%%%%%%%
%>> MOVE TO APPENDIX?
\subsection{Complexity measures}%{Gaussian width and Rademacher complexity}
Here we recollect three %two 
useful measures of the complexity of a set or function class.
\subsubsection{Gaussian Width}
\begin{definition}[Gaussian Width]
The Gaussian width of a bounded set $T$ is defined as:
\begin{equation}
w(T) := \E_{g\sim \N(0,I)} [\sup_{x\in T} g^T x]
\label{eq:gw}
\end{equation}
\end{definition}
This quantity originates from metric geometry, and it measures to what extent the set of points in $T$ is similar to a standard Gaussian distribution in $\R^d$ or, more informally, to what extent $T$ simulates white noise. Thus Gaussian width is a measure of complexity. %natural measure of compressibility and the smaller the Gaussian width of a set the more compressible the set is. % by RP
We note the basic properties that if $T' \subseteq T$ then $w(T') \leq w(T)$ and also that the Gaussian width of a set is invariant under translation by any fixed vector $v$, that is $w(T) = w(T+v)$ where the plus sign denotes Minkowski addition. For more details on this notion, its properties and its variants, we refer the interested reader to \cite{Vbook} (Sec. 7.5) and references therein. A related quantity defined as in equation \eqref{eq:gw} but with the supremum on RHS taken over $|g^{T}x|$ has been used already in statistical learning theory under the name of Gaussian complexity, where it is often employed to measure the complexity of a function class, as is the closely related notion of Rademacher complexity \cite{KP02,BarMen02}.
%\newpage
%LIST PROPERTIES THAT WE USE
\subsubsection{Rademacher Complexity}
\begin{definition}[Empirical Rademacher Complexity]\label{rc} 
The empirical Rademacher complexity of a function class $F=\{ f: \X\rightarrow \R\}$, w.r.t. a set of evaluation points $S_N=\{x_1,...,x_N\}\in \X^N$ is defined as:
\begin{equation}
\hat{\rad}_{N}(F) := \E_{\gamma} \left [\sup_{f\in F} \frac{1}{N}\sum_{n=1}^N \gamma_n f(x_n)   \right]
\end{equation}
where $\gamma=(\gamma_1,...,\gamma_N)$ with entries drawn i.i.d. from $\{-1,1\}$ with equal probability.
%If the training sets of observations of size $N$ are drawn i.i.d from some distribution $X_i \sim \mathcal{D}$, $\forall i$ then the \emph{Rademacher Complexity} is given by $\rad_{N}(F):=\E_{\mathcal{D}}\left[ \hat{\rad}_{N}(F)\right]$.  % we are not using this
\end{definition}
Definition \ref{rc} is the version from \cite{Mohri}, with the absolute value omitted assuming that $F=-F$.

The empirical Rademacher complexity has been used to derive uniform risk bounds, in particular the following theorem \cite{KP02,BarMen02} is a classic (Theorem 3.1 in \cite{Mohri}).
\begin{theorem}[Bartlett and Mendelson \cite{BarMen02}] Let $\hypclass=\{h: \X \rightarrow [0,1]\}$ be a function class.
For any choice of loss function $\ell$, $\forall \delta>0$, and for all $h\in \hypclass$ it holds uniformly with probability at least $1-\delta$:
\bqn
\E_{x,y}[\ell(h(x),y)] \le \frac{1}{N}\sum_{n=1}^N \ell(h(x_n),y_n) + 2\hat{\rad}_N (\ell\circ \hypclass)
+3 \sqrt{\frac{\log(2/\delta)}{2N}}\nonumber
\eqn
\end{theorem}

\subsubsection{Vapnik-Chervonenkis (VC) dimension} %-- SEC TO BE COMPLETED -- 
\begin{definition}[VC dimension]
The VC dimension of a function class $\hypclass=\{h: \X \rightarrow \{-1,1\}\} $
is the cardinality $V$ of the largest set of points from $\X$ %that can be shattered by elements of the class.
that can be labelled in all $2^V$ possible ways by elements of $\hypclass$. 
\end{definition}

The VC dimension is a standard complexity measure for binary valued functions introduced in the works of Vapnik and Chervonenkis \cite{Vapnik}, used to derive uniform bounds for classification. The following improved version is due to \cite{BarMen02}.
\begin{theorem}[Bartlett and Mendelson \cite{BarMen02}]
Let $\hypclass=\{h: \X \rightarrow \{-1,1\}\} $ be a binary valued function class with 
VC dimension $V$.
For any $\delta>0$, and for all $h\in \hypclass$ it holds uniformly with probability at least $1-\delta$:
\bqn
\pr_{x,y}[h(x)y \le 0] \le \frac{1}{N}\sum_{n=1}^N \Eins(h(x_n)y_n \le 0) + 
c\sqrt{\frac{V+\log(1/\delta)}{N}}
\nonumber
\eqn
where $c$ is an absolute constant independent of the dimensionality of $\X$.
\end{theorem}

%%%%%%%%%%%%%%%%%%%%%%%%%%%%%%%%%%%%%%%%%%%%%%%%%%%%%%%%%%%%%%%%%%%%%%%%%%%%%%%%%%%%%%%%%%%
\section{Risk bounds for compressive linear ERM classifiers}\label{sec:RP}
%%%%%%%%%%%%%%%%%%%%%%%%%%%%%%%%%%%%%%%%%%%%%%%%%%%%%%%%%%%%%%%%%%%%%%%%%%%%%%%%%%%%%%%%%%%
We start by providing improved versions of previous results on bounding the generalization of randomly projected ERM classifiers in \cite{icml13}. 

\begin{theorem} \label{thm:riskb2}
Take any $h\in \R^d$.
Let $R$ be a $k\times d$ subgaussian random matrix with i.i.d. entries, $k\le d$, and denote %by $\theta_u^h$ the angle between two vectors $u,x\in\R^d$, and by 
$f_k^+(\theta_u^h):= f_k(\theta_u^h) \cdot\Eins(h^Tu>0)$, where $f_k(\theta_u^h)=\pr_R\ny h^TR^TRu \le 0 \zr$. 
For any $\delta\in(0,1)$, the following holds for the compressive ERM classifier $\erm_R$ with probability $1-2\delta$:
{\small
\begin{eqnarray*}
\pr_{x,y} [(\hhr^T Rx+b) y\le 0] &\le&  
\frac{1}{N}\sum_{n=1}^N \Eins(h^Tx_n y_n \le 0 )   +c\sqrt{\frac{k+1+ \log(1/\delta)}{N}} 
...\\
&+&  \frac{1}{N}\sum_{n=1}^{N} f_k^+(\theta^{h}_{x_ny_n})
+ \min \left \{
\frac{1-\delta}{\delta} \cdot \frac{1}{N}\sum_{n=1}^{N} f_k^+(\theta^{h}_{x_ny_n}),
\sqrt{\frac{1}{2}\log\frac{1}{\delta}}
\right \}
\label{eq:bound1}
\end{eqnarray*}}
Likewise, with probability $1-2\delta$:
{\small
\begin{eqnarray*}
\pr_{x,y} [(\hhr^T Rx+b)y \le 0] &\le & 
 \pr_{x,y} [h^{*T}xy \le 0]  
+2c\sqrt{\frac{k+1+ \log(1/\delta)}{N}}
...\nonumber\\
&+&\E_{x,y}[f_k^+(\theta^{h^*}_{xy})]
+ \min \left \{ \frac{1-\delta}{\delta} \cdot \E_{x,y}[f_{k}^+(\theta^{h^*}_{xy})],
\sqrt{\frac{1}{2}\log\frac{1}{\delta}}
\right \} \label{eq:bound2}
\end{eqnarray*}}
where $c$ is an absolute constant independent of the dimensionality of $\X$.
\end{theorem}

Observe that %the first two terms taken together in the right hand sides of the above bounds are the same as the VC-bound for $k$-dimensional linear classification -- that is, % no, because $h$ is d-dim.
the VC complexity term is greatly reduced at the price of the last two terms that represent the error incurred by working with a random compression of the data. As $k$ grows towards $d$, then the $f_k^+(\theta)$ terms decrease towards zero, resembling the classical VC bound for the original $d$-dimensional ERM classifier. 
However, in some cases the generalization error upper bound for learning from the RP data can also be smaller than for learning from the original data -- for example, if the last two terms are negligible for some value of $k < d$ -- making it plausible that the compressed classifier may generalize better than its dataspace counterpart. % Hmmm... that would be really interesting. But might be hard to prove! Q: Would any of the distributions in the construction of the VC lower bound fit this bill?
Equivalently, the compressed classifier would then require a smaller sample complexity for the same fixed error guarantee as its dataspace counterpart. 

 For data separable with a margin $\min_n \cos(\theta_{x_ny_n}^h)>0$, and any small $\epsilon>0$, we can make the latter two terms smaller than $\epsilon$, using Lemma \ref{thm:flip2}, by setting $k\ge \frac{8\log(1/(\epsilon\delta))}{\min_n \cos^2(\theta_{x_ny_n}^h)}$. 
For data with zero margin, it is easy to modify the proof of Theorem  \ref{thm:riskb2} by introducing a margin parameter $\gamma>0$, yielding the following variation: 
\begin{corollary}\label{riskb2corr}  Fix some $\gamma>0$.
Take any $h\in \R^d$.
Let $R$ be a $k\times d$ subgaussian random matrix with i.i.d. entries, $k\le d$, and denote 
$f_k^{\gamma}(\theta_u^h):= f_k(\theta_u^h) \cdot\Eins(\cos(\theta_u^h)>\gamma)$, where $f_k(\theta_u^h)=\pr_R\ny h^TR^TRu \le 0 \zr$. 
For any $\delta\in(0,1)$, the following holds for the compressive ERM classifier $\erm_R$ with probability $1-2\delta$:
{\small
\begin{eqnarray*}
\pr_{x,y} \{(\hhr^T Rx+b) y\le 0\} &\le&  
\frac{1}{N}\sum_{n=1}^N \Eins\{\cos(\theta_{x_ny_n}^h) \le \gamma \}   +c\sqrt{\frac{k+1+ \log(1/\delta)}{N}} 
...\\
&+&  \frac{1}{N}\sum_{n=1}^{N} f_k^{\gamma}(\theta^{h}_{x_ny_n})
+ \min \left \{
\frac{1-\delta}{\delta} \cdot \frac{1}{N}\sum_{n=1}^{N} f_k^{\gamma}(\theta^{h}_{x_ny_n}),
\sqrt{\frac{1}{2}\log\frac{1}{\delta}}
\right \}
\label{eq:bound1}
\end{eqnarray*}}
where $c$ is an absolute constant.
\end{corollary}
Now by similar argument, setting $k\ge \frac{8\log(1/(\epsilon\delta))}{\gamma^2}$ is sufficient to guarantee that the last two terms 
in Corollary \ref{riskb2corr} are below $\epsilon$. 
In Section \ref{sec:struct} we will identify some sufficient conditions, i.e. fortuitous structure possessed by the data, under which this bias  $\epsilon$ vanishes.% with high probability.

Another variant of Theorem \ref{thm:riskb2} is when Gaussian RP is used -- the bounds %of Theorem \ref{thm:riskb2}
can be tightened in that case, as the availability of the exact expression for $f_k(\cdot)$ allows us to include in our bound not only the probability that predictions flip from correct to incorrect after RP, but also the probability that they flip back from incorrect to correct.

\begin{theorem}\label{riskb}
Take any $h\in \R^d$.
For all $\delta\in(0,1)$, the following holds for the compressive ERM classifier $\hhr$ with probability $1-2\delta$:
{\small
\begin{eqnarray*}
\pr_{x,y} \{(\hhr^T Rx+b) y \le 0 \} &\le& 
 \frac{1}{N}\sum_{n=1}^{N}f_{k}(\theta^{h}_{x_ny_n}) +c\sqrt{\frac{k+1+ \log \frac{1}{\delta}}{N}}\\
&+& \min \left \{
\frac{1-\delta}{\delta} \cdot \frac{1}{N}\sum_{n=1}^{N}f_{k}(\theta^{h}_{x_ny_n}),
\sqrt{\frac{1}{2}\log\frac{1}{\delta}}
\right \} %\label{eq:bound1G}
\end{eqnarray*}
and likewise,
\begin{eqnarray*}
\pr_{x,y} \{(\hhr^T Rx+b)y\le 0 \} &\le& 
\E_{x,y}[f_{k}(\theta^{h^*}_{xy})] +2c\sqrt{\frac{k+1+ \log \frac{1}{\delta}}{N}}\\ 
&+& \min \left \{
\frac{1-\delta}{\delta} \cdot \E_{x,y}[f_{k}(\theta^{h^*}_{xy})],
\sqrt{\frac{1}{2}\log\frac{1}{\delta}}
\right \}   %\label{eq:bound2G}
\end{eqnarray*}
}
where $c$ is an absolute constant independent of the dimension of $\X$.
\end{theorem}
The proofs of Theorems \ref{thm:riskb2} and \ref{riskb} are given in Section \ref{pf:compr}.

%%%%%%%%%%%%%%%%%%%%%%%%%%%%%%%%%%%%%%%%%%%%%%%%%%%%%%%%%%%%

\subsection{Identifying benign structure}\label{sec:struct}
The bounds in the previous section show that good generalization performance can be obtained from a classifier trained on
randomly projected data, provided that the data have some structure which keeps the probability of label flipping low. 
From \cite{icml13}, it was already evident that these structures include the cases when data classes are separable or soft-separable with a margin. Identifying other structural properties of data that imply a low flipping probability was left for future work.

In this section we use the notion of Gaussian width to capture the effects on the data support's geometry on the performance of the compressive classifier. In particular, we require that the flipping probability is uniformly below some $\delta > 0$.\\
We note that if a classification problem is inherently incompressible then such guarantee can only hold for all $\delta > 0$ by setting $k=d$. However data frequently do possess compressible structure even if in many high-dimensional settings we may not know precisely what it is.

Although the quantities that will appear in the bounds of this section are usually unknown in practice, our goal here is to provide a better theoretical understanding of when and why compressive classification performs well, as a consequence of the geometry of the problem. 

From a generic result we will recover, as special case corollaries, results that so far have only been proved for the compressive Fisher Linear Discriminant classifier, namely that %--- begin cut -- 
%the generalization error depends on the `effective dimension' of the data support and not on $d$ \cite{alt13}, that -- end cut -- 
for multi-class 1-vs-all classification the compressed dimension only needs to grow as $\Omega(\log \# \text{ classes})$ \cite{kdd10} -- in sharp contrast with $\Omega(\log N)$ in some earlier work \cite{Arriaga} --
and there is no specific sparse representation requirement for these classification guarantees, although we shall see that sparsity is one of a variety of structures that are benign for compressive classification.

Since we work with the zero-one loss directly, we can take all $h\in \hypclass$ and $x\in \X$ to have unit norm without loss of generality. Now let $U := \ny \frac{xy}{\norm{x}}: x\in{\mathcal X}, y\in\{-1,1\}\zr$ and for an arbitrary (but fixed) $h\in \hypclass$, define the margin of $h$ to be
$\gamma_h:= \inf_{u\in U} \cos(\theta_u^h)$. We further define, for any fixed $\gamma$ and $h$, the following set:
\begin{eqnarray}
T_{h,\gamma}^+ &:=& \ny u\in U: \cos(\theta_u^h)\ge\gamma \zr \subset S^{d-1}; \text{~where } \gamma>0
\label{eq:set}
\end{eqnarray}
Thus the set $T_{h,\gamma}^+$ defined above is the set of all points in the support $\X$ of the underlying (unknown) input data distribution that the high dimensional vector $h\in\hypclass$ classifies correctly with a pre-specified margin $\gamma$.
Note, however, that we do not require the data support of the classes to have a margin, or to be linearly separable.

With these definitions in place, the following theorem gives a condition on the compressed dimension $k$ that ensures a risk guarantee for the compressive linear ERM classifier with a similar flavour to margin or VC- type bounds for a $k$-dimensional dataspace classifier. This condition depends on the geometric structure of the problem, as reflected by the Gaussian width of the set $T_{h,\gamma}^+$.

\begin{theorem}\label{thm:uf}
Take any $h\in \R^d$. Let $R$ be a $k\times d, k\le d$ isotropic subgaussian random matrix with 
independent rows each having subgaussian norm  bounded as $\norm{R_i}_{\psi_2}\le K$.
Fix some $\gamma > 0$ large enough that $\gamma \ge \gamma_h$. 
Then, for any $\delta>0$ there are absolute constants $C,c > 0$ such that with probability $1-2\delta$ the generalization error of the compressive ERM classifier, $\erm_R$, is bounded as the following:
{\small
\begin{equation}
\pr_{x,y}[(\erm_R^T Rx+b)y\le 0] \le
\frac{1}{N}\sum_{n=1}^N \Eins\(\cos(\theta_{x_ny_n}^h) < \gamma\)
+c\sqrt{\frac{k+1+\log(1/\delta)}{N}}  
\end{equation}
}
provided that 
$k \ge 
CK^4\(w(T_{h,\gamma}^+) + \sqrt{\log(1/\delta)}\)^{2}\gamma^{-1}$.
\end{theorem}

The role of $\gamma$ is similar to that of the margin parameter in margin bounds in the sense that, if $h$ has positive margin then the bound is tight, otherwise it is less tight, but the bound holds nevertheless.

In the special case when $h$ has a margin and %$\gamma_h = \gamma$ 
$T_{h,\gamma}^+ = T_{h,\gamma_h}^+$ then the empirical risk on the compressive classifier for the required $k$ incurs no increase in error from the compression, and has a guarantee as good as that for a $k$-dimensional dataspace classifier.
We find it interesting at this point to remark that -- as we shall see from the proof -- here the concept of margin arises as a natural way to satisfy the requirement to reduce flipping probabilities, rather than appearing `rabbit-out-of-hat' as an (arguably) artificial-looking precondition.
That is unlike previous work, such as \cite{Arriaga,Balcan}, which assumes the existence of a margin and proceeds to quantify by what extent RP shrinks it, here we see that margin arises as one sufficient condition for compressive classification to work well.

On the other hand, if $h$ has no positive margin, then $\gamma_h=0$, but $\gamma>0$ ensures that the bound still holds. 
Now $\gamma$ takes the role of a parameter that we can choose in order to control the trade-off between the bias introduced by working with the randomly compressed data, and some practically desirable target dimension $k$ required by a particular task.

%The statement holds for an arbitrary $h\in\hypclass$ -- we can choose for instance the empirical minimiser in the data space to directly relate the ERM classifiers in the two spaces.% not 100% sure (choice of $h$ can depend on the sample?) 

Most importantly, observe that the bound of Theorem \ref{thm:uf} does not depend on the dimension $d$ of the original data. 

\begin{corollary}\label{mlb1} Suppose we have an $m$-class classification problem with examples to be classified using a one-vs-all scheme of linear ERM classifiers, and suppose further that for each such binary ERM classifier the condition on the margin parameter of Theorem \ref{thm:uf} holds. Then, $\forall \delta > 0$, w.p. $1-2\delta$, the sufficient projection dimension is lower bounded as:
\bqn
k %&\ge& CK^4 \(w(\max_{i\in {\mathcal C}}\supp( Class_i)+\sqrt{\log(m/\delta)}\)^2/\gamma \label{mlb2}\\
&\ge&
CK^4 \( \max_{i\in {\mathcal C}} w( T_{h,\gamma,i}^{+})+\sqrt{\log(m/\delta)}\)^2 \gamma^{-1} \nonumber
\eqn
\end{corollary}
Indeed, since in the one-vs-all scheme \cite{Rifkin} one builds a binary classifier for each class, for $m$ classes we require the error bound in Theorem \ref{thm:uf} to hold w.p. $1-2\delta/m$ simultaneously for each class. Hence plugging this into the $k$-bound condition of %Corollary \ref{c:uf} and/or of 
Theorem \ref{thm:uf} yields Corollary \ref{mlb1}. 
% *** The more classes there are to discriminate between, the less likely it is (absent structure in the data generator) that any one class is linearly separable from all others. Thus the condition on the margin parameter is likely to be trivially satisfied in practice - I imagine the Gaussian width term probably grows a fair bit though? Or at least has the potential to?
% or (\ref{mlb2}) respectively.
%http://www.jmlr.org/papers/v5/rifkin04a.html

Also, if we apply the one-vs-one multi-class classification scheme instead, then we build $\binom{m}{2} = \frac{m(m-1)}{2}$ 
binary classifiers, but still would obtain a logarithmic dependence for the required $k$ on the number of classes. %Note that that $\max_{\in {\mathcal C}}w(T_{h,\gamma,i}^+)$ never exceeds $\sqrt{d}$, since $T_{h,\gamma,i}^+ \subseteq S^{d-1}, \forall i=1,...,m$
%all of these sets are subsets of the $d$-dimensional unit sphere. 

Previously, the fact that the target dimension for  
$m$-class compressive linear classification only needed to be of order $\log m$ was only known in the special case of the compressive FLD classifier \cite{kdd10}. % -- and this was in sharp contrast with the $\log N$ scaling of $k$ in early compressive learning results in the literature (e.g. \cite{Arriaga}).
Corollary \ref{mlb1} %-(\ref{mlb2}) 
generalizes this result to linear ERM classifiers. % Is there more than 1 such classifier?

It should be noted that the expression of the lower bound on $k$ %$$\in{\Omega(w(T_{h,\gamma}^+)^2/\gamma)}$
is a function of the underlying geometry of the problem. In particular, geometric structures of the data support that have low Gaussian width are benign for compressive  classification. 
Some examples are given below.

\subsubsection{Examples of benign geometry for 2-class classification}
As $T_{h,\gamma}^+\subseteq S^{d-1}$ we always have $w(T_{h,\gamma}^+)^2\le d$. This follows e.g. from Prop. 7.7.2 in \cite{Vbook}. Equality holds if $T_{h,\gamma}^+$ is the whole $d$-dimensional hypersphere, and it can be much less than $d$ when there is structure in the data support. 
\begin{itemize}
\item If $h$ has a large margin for the points of the data support that don't contribute to the empirical error term, then $w(T_{h,\gamma}^+)$ reduces. 
To see this note that if the correctly classified datapoints are concentrated around the antipodes of the unit sphere and $h$ is roughly in the direction of the north pole (say) -- so that it has a large margin -- then $T_{h,\gamma}^+$ is contained in a spherical cap making a small angle $\phi:=\arccos(\gamma_h)$ at the origin, while if the margin with $h$ is small then $\phi$ for the spherical cap containing $T_{h,\gamma}^+$ is larger.
%To see this, take the larger of the minimum enclosing balls of the support of each class. Then $T_{h,\gamma}^+$ consists of two, possibly overlapping hyperspherical caps of the $d$-dimensional hypersphere. The farther the classes are from each other the smaller the angle, $\phi:=\arccos(\gamma_h)$, of these spherical caps. 
As shown in \cite{Amelunxen, Bandeira}, the squared Gaussian width of a spherical cap is $d(1-\cos^2(\phi))+O(1)$.
%$d \times \sin^2(\phi)+O(1)$. 
Observe also that the $h$ that reduces this quantity also increases the margin, $\cos(\theta_{xy}^h)$. 
\item If the data lives in an $s$-dimensional subspace, then  $[w(T_{h,\gamma}^+)]^2$ is of order of $s$ rather than of $d$. Note also the trade-off for choosing $h$: $h$ can be chosen to be sparse and this will reduce the contribution of the Gaussian width term to the error bound. However if the data live outside the modelled subspace then the empirical error term is likely to increase in response. Interesting to note, in the case when irrelevant noise features exist, then even a well-chosen sparse $h$ cannot completely circumvent their bad effect on compressive classification -- even though $[w(T_{h,\gamma}^+)]^2$ is reduced in this way, the noise components increase $\norm{x}$ in the denominator of the cosine, so the empirical error term will still tend to increase. 
%\item If the data support does not fill the full space in the sense that it has small eigenvalues in some directions, then the squared Gaussian width of the support of the minimum enclosing ellipse of each class is upper bounded by the trace of the covariance matrix defining the ellipse (easy to show), so $[w(T_{h,\gamma}^+)]^2$ will be smaller than that of the union of the two ellipses that contain the classes.
\item %If the data classes are separated by living in disjoint  $s$-dimensional subspaces -- OOPS, NO, THAT'S A NONLINEAR PROBLEM, BUT:
If the data support has an $s$-sparse representation, or it lives in a union of disjoint $s$-dimensional subspaces, then the squared Gaussian width of the support of the classes is of order ${\Theta}(s\log(2d/s))$, cf. Lemma 3.5 in \cite{PlanVershynin}. Here we use the facts that the projection of the data support onto the unit sphere is then a union of disjoint $(s-1)$-spheres, the Gaussian width of a sphere is the same as that of a ball, and $T_{h,\gamma}^+$ is a subset of the projection of the support onto the unit sphere, so its squared Gaussian width $[w(T_{h,\gamma}^+)]^2$ is no larger.
\enlargethispage{\baselineskip}
% c\sqrt{s\log(2d/s)}\le w_1(K) \le C\sqrt{s\log(2d/s)}
% where w_1(K) := \E_g\sup_{x,y\in K; \norm{x-y}_2\le 1}
%\item Width of approximately sparse vectors, modelled as living in an $\ell_1$ ball is given in Plan & Vershynin & Yudovina, sec. 2.6. Is it applicable here? Probably not(?)
\end{itemize}

More examples of structured sets exist of course, there are many ways in which the data may not `fill' the full $d$-dimensional space, and our generic bound captures the effect these structures have on the compressive classifier performance. This could explain, for instance, why a drastic random compression still works surprisingly well in practice for face recognition \cite{Goel} but not so well on difficult data sets that contain large amounts of unstructured noise \cite{Xie}.\\
In practice, if we have some knowledge of the underlying structure of the data support this can be exploited, and our bound provides guidance on how much we can compress and the likely benefits or costs (in terms of classification accuracy) of doing so. On the other hand, in many modern problem domains -- especially high-dimensional ones -- such prior knowledge may be weak or unavailable. %However, we may still wish to use classification, for example either as an exploratory tool or as an early step towards inference. Thus in the next section we consider how to discover and exploit such unknown benign geometry whenever it is present.
However, the bounds of Theorems \ref{thm:riskb2}-\ref{riskb} and Corollary \ref{riskb2corr} adapt to such structures through the sample average of pairwise flipping probabilities, which have the advantage that they can be estimated from the training set. In the next section we take this novel tool forward to tighten classical uniform bounds on the original dataspace classification problem.
%------------------------------

\section{From random projections back to the dataspace: Geometry-aware error bounds with the zero-one loss}\label{sec:Dataspace}
With the insights gained regarding the ability of RP to exploit benign geometry for compressed classification, %in the sense that the smaller the required target dimension the smaller the required sample size, 
it is natural to ask if a similar approach is possible for the original high dimensional classification problem.

In this section we propose a generic principle to discover and exploit benign geometry inherent in data by means of the distortion that RP incurs on the loss function of a learning problem. This gains us access to the effects of small perturbations on low complexity sets even for scale-insensitive losses. Using this idea we provide further new error bounds for binary linear classification, which are able to exploit naturally occurring structure while continuing to work directly with the zero-one loss. Our bounds are data dependent, and remain informative in small sample conditions. 
%Depending on the choice of target dimensions we present a global and a local variant, which we then instantiate in several ways to draw connections between existing classification approaches, including two different explanations of boosting. Our bounds potentially can inspire new algorithms and we present preliminary results using one such approach as proof of concept.
Depending on different criteria for the choice of projection dimension we present several variations, which we instantiate to draw connections between existing classification approaches, including two different explanations of
boosting. Finally, we empirically demonstrate that our general bound is informative enough in practice to serve as the objective function for learning a classifier.

\subsection{Rationale}
A fundamental result in statistical learning theory is that the VC dimension completely characterizes PAC-learnability under mild measurability conditions \cite{Vapnik,Blumer} -- indeed, its guarantees hold for any distribution with bounded support, and upper and lower bounds agree up to logarithmic factors \cite{AnthonyBartlett}. %, which means that major improvements are not possible. 
Yet this theory is often found to be uninformative in practice because its guarantees are frequently too pessimistic and do not agree with practical experience. 

The practical inutility of VC bounds is due to their generality, in particular their insensitivity to any benign geometry in natural data sets and learning problems. 
%Various scale-sensitive generalizations of the VC dimension have been proposed, and 
Rademacher complexity bounds \cite{BarMen02,KP02}
have been developed towards resolving this issue with the use of scale-sensitive surrogate loss functions. However surrogate losses, for instance the margin loss, \emph{pre-define} what geometry is benign for the problem. In consequence the resulting bounds are more informative than VC bounds only when the corresponding geometry is present in the data. 

%In this work we put forward the idea to capture the inherent benign geometry by means of the distortion that random projections (RP) incur on the loss function of a learning problem. 
Results with a similar flavour for RP, that is which pre-define certain properties the data generator must possess and use RP as a theoretical device for developing risk bounds, have already been proposed in some forms in \cite{Garg,Garg03}, \cite{Balcan}, and \cite{Arriaga}. For example, the latter in a very restricted setting that assumes linear separability of the data under the presence of a margin. Here in turn we make no assumptions on the data generator, but rather than replacing the zero-one loss with some (of the many) surrogate loss functions and pursuing a random projection of the function values only, we will measure the complexity of the function class -- similarly to analyses using Rademacher complexity --
and approach structure discovery directly by RP of the input space. Scale-sensitivity then will appear naturally as a by-product, and we obtain a generic principle by which the generalization error in a high dimensional learning problem can be controlled by a distortion term, and the error of a low dimensional version of the same problem. 
Thus, in this approach, the complexity of the problem depends on its compressibility, which we have already seen depends on the presence of benign structure in the data.
Conversely, we can \emph{choose} the model complexity according to the available sample size, and observe the distortion incurred directly from the sample.
Different criteria to choose the target dimension give rise to different variations of our bounds, which turn out to reveal new connections between existing successful classification approaches, including the recently proposed Margin Distribution Machine as well as yielding two different explanations of Boosting.
We also speculate that bounds such as ours have the potential to generate new classification algorithms and we will demonstrate a proof-of-concept to this effect, from first principles, by turning a bound into a classifier and comparing it to the gold-standard Support Vector Machine (SVM). 

%------------------------------

\subsection{New bounds on dataspace classification}
Let $R$ be a $k\times d, k\le d$ random matrix with i.i.d. zero-mean Gaussian entries, that is $R$ is an RP matrix as before but here it will serve only as an analytic tool, and is not used explicitly for dimensionality reduction.
 
As before, denote by $\theta_u^h$ the angle between the $d$-dimensional vectors $u$ and $h$,  and define the function $f_k(\theta_{xy}^h):= 
%\ell((Rh)^TRx,y)=
\pr_R((Rh)^TRxy \le 0)$.
We have the following result:
\begin{theorem}\label{DataspaceB} Fix any positive integer $k \le d$.
For any $\delta>0$, with probability at least $1-\delta$ with respect to the random draws of $\train$ of size $N$,
$\forall h\in \H$ the generalization error of $h$ is upper bounded as the following:
\begin{eqnarray}
\pr_{x,y}[h^T\phix y \le 0]&\le& 
\frac{1}{N}\sum_{n=1}^N \min(1,2f_k(\theta_{\phix_n y_n}^h))+
\frac{2\sqrt{2}}{\sqrt{\pi}}\sqrt{\frac{k}{N}}+
3\sqrt{\frac{\log(2/\delta)}{2N}}\nonumber
\end{eqnarray}
\end{theorem}
The first term is the sum of the empirical zero-one error and twice the average probability that training points classified correctly by $h$ become misclassified following a random projection. This latter component measures the distortion caused by RP, and captures the benign geometry for the problem: A data distribution for which the loss function  withstands the perturbation of a RP is benign.

The second term is analogous to the complexity term in VC bounds -- however the target dimension $k$ of the RP takes the place of the VC dimension. This term is increasing in $k$.
%Notice the trade-off: 
However, the same parameter $k$ also appears in the first term, playing a role that may be thought analogous to an inverse margin, and this term is decreasing in $k$. Thus the first two terms capture a trade-off between the complexity of the model and the complexity of the data. Note also that as $k \rightarrow d$ the first term vanishes and we recover a standard VC bound.
Most importantly, while traditional VC bounds use worst-case complexity and in consequence prescribe very large sample sizes that most often cannot be met realistically in practice, here instead we can choose $k$, as an `affordable complexity', depending on the available sample size. The distortion term measured on the sample will then reflect the extent of error incurred.

Figure \ref{synth} demonstrates the bound of Theorem \ref{DataspaceB} a synthetic example, with $\delta=0.05$. The sample size was $N=5000$, and the cosine values were generated from a 0-mean Gaussian with variance $1/9$; values outside of $[-1,1]$ were replaced with samples from the uniform distribution on $[-1,1]$. The first term of the RHS of the bound (`Flip p') is plotted against the sum of the last two terms (`Complexity'), along with their sum (`Bound'). 
We see the trade-off between the average flip probability under RP, and the complexity of the function class in the RP space, as the RP dimension $k$ varies. 
\begin{figure}
\centering
\includegraphics[width=7cm,height=6cm]{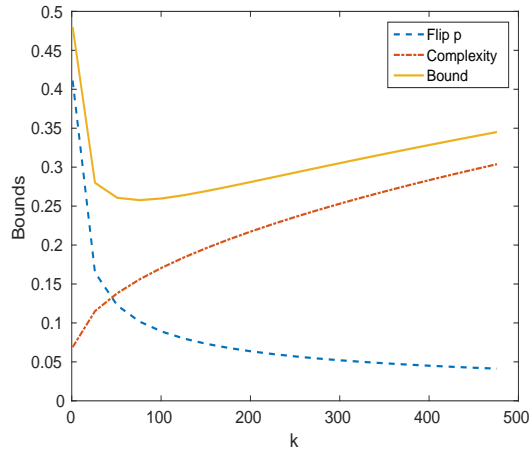}
\caption{\label{synth}\vspace{0.2cm} Illustration of the bound of Theorem \ref{DataspaceB}.}
\end{figure}

In the form stated, $k$ needs to be specified before seeing the sample. Alternatively we may apply structural risk minimization (SRM) \cite{Vapnik} to make the bound uniform over all values of $k$. Take the sequence $(k_i)_i\ge 1, k_i=i$, and let $\mu_i$ -- chosen before seeing the sample -- 
be our prior belief in the value $k=i$ s.t. $\sum_{i\ge 1}\mu_i=1$. Then applying Theorem \ref{DataspaceB} with 
$\delta_i:= \delta\mu_i$, and applying union bound over the sequence of values $k_i$, we get the conclusion of Theorem \ref{DataspaceB} simultaneously for all $k$ at an expense of an small additive error of $\sqrt{\frac{\log(1/\mu_i)}{2N}}$.
For instance, if we choose the exponential prior probability sequence $\mu_i=2^{-k}$, then this additional error term evaluates to   
$3\sqrt{\frac{\log(2)}{2}} \sqrt{\frac{k}{N}}$.

The function $f_k(\theta)$ that appears in the first term 
was previously used to bound the error of compressive classifiers in the previous section, and both the exact form and a tight analytic upper bound are available.% as the following, with a new and much simpler proof given in the Supplementary.
In the context of dataspace classification, $f_k(\theta)$ is a kind-of `induced' margin loss. 
Figure \ref{fkth} gives a graphical illustration of this function against the normalized margin $\cos(\theta)$. 
However, in contrast to a pre-defined margin loss, this function comes from a more generic principle as applied here to the zero-one loss, and it is not limited to margins.
%We shall see several further manifestations of the same principle in later subsections.
%\begin{wrapfigure}{r}{0.5\textwidth}
\begin{figure}
  \begin{center}
\includegraphics[width=7cm]{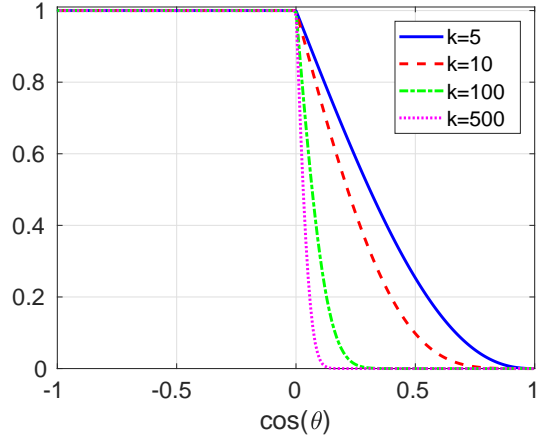}%{fkth.eps}
\caption{\label{fkth}The function $\min(1,2f_k(\theta))$, as a function of $\cos(\theta)$, in the role a classification loss function.}
  \end{center}
\end{figure}
%\end{wrapfigure}

%%%%%%%%%%%%%%%%%%%%%%%%%%%%%%%%%%%%%%%%%%%%%%%%%%%%%%%%%%%%
%========================================================
\subsection{Relating our theory to existing work and uncovering new connections}
%=========================================================
A desirable property for new theoretical results is an ability to provide a unifying context that can connect earlier work. In this section we show that our theory is able to draw new connections between existing methods.
We start by creating a variation on our dataspace bound that is an upper bound on the original but allows the values of $k$ to differ for each input point.
%We can allow the values of $k$ to differ for each input point. Since RP serves merely as an analytic tool here and is not part of an algorithm, this approach is practical, and gives us choices to pre-define more compressible and less compressible regions of the input space when appropriate. We will also use such a variant of the bound to draw new connections between existing methods in the next section.

\begin{theorem}\label{DataspaceB_kx}
Let $k: \R^d \times \H \rightarrow \mathbb{N}$ a deterministic function specified independently of $\T^N$. Then $\forall h\in \H$, with probability $1-\delta$ with respect to the random draw of a training set of size $N$,
%we have for any choice of positive integers $k=k(x_ny_n,h)$,
the generalization error of $h$ is upper bounded as the following:
\begin{eqnarray}
\pr_{x,y}[h^T\phix y\le 0]&\le &
\frac{1}{N}\sum_{n=1}^N \min(1,2f_{k(x_ny_n , h)}(\theta_{\phix_n y_n}^h))
+
2\sqrt{\frac{2}{\pi}} 
\sqrt{\frac{1}{N}\max_{n=1}^N k(x_ny_n , h)} \nonumber\\
&+& 3\sqrt{\frac{\log(2/\delta)}{2N}} +
3\sqrt{\frac{\log(2)}{2}}\sqrt{\frac{1}{N}\max_{n=1}^N k(x_ny_n , h)}
\nonumber \end{eqnarray}
\end{theorem}
%

%=======================================================
\subsubsection{Connection with the Large Margin Distribution Machine}\label{sec:exp}

In this section we will instantiate the bound of Theorem \ref{DataspaceB_kx}. The idea is to define the function $k=k(xy, h)$ in a way to lead to an analytic convex expression of $h/\norm{h}$. %Dependence of this function on $h$ is allowed since Theorem \ref{DataspaceB_kx} holds uniformly for all values of $k(\cdot)$ that occur in the bound. 

To do this, first we bound the $f_k(\theta)$ term using the analytic upper bound given in eq.(\ref{ball}), which is tight on the interval 
$\theta\in [-\pi/2,\pi/2)$, and is still a bound on $\theta\in [\pi/2,3\pi/2]$.
\begin{eqnarray}
\min(1,2f_k(\theta)) \le 2\exp
\ny -\frac{k \cos^2(\theta) \cdot \sgn(\cos(\theta))}{2} \zr
\end{eqnarray}
For $k(\cdot)$ we choose the following:
\begin{eqnarray}
k(xy,h) := %\frac{2\norm{xy}}{\abs{\cos(\theta_{xy}^h)}}
         %= \frac{2\norm{x}}{\abs{\cos(\theta_{x}^h)}} = \frac{2\norm{h}}{\abs{h^Tx}}
     \frac{2}{\abs{\cos(\theta_{xy}^h)}}
\end{eqnarray}
Plugging this into Theorem \ref{DataspaceB_kx} we obtain:
\begin{corollary}\label{DataspaceB5}
With probability $1-\delta$ w.r.t. the training set of size $N$,
$\forall h\in \H$, 
\begin{eqnarray}
\pr_{x,y}[h^Txy\le 0] &\le&\frac{1}{N}\sum_{n=1}^N
2\exp\(-\frac{h^Tx_ny_n}{\norm{h}\cdot\norm{x_n}}\) + \frac{4}{\sqrt{\pi}}
\frac{1}{\sqrt{N}}\cdot \max_n \sqrt{\frac{\norm{h}\cdot\norm{x_n}}{\abs{ h^Tx_n}}
}\nonumber \\
&+& 3\sqrt{\frac{\log(2/\delta)}{2N}} + 3\sqrt{\frac{\log(2)}{N}}\cdot \max_n \sqrt{\frac{\norm{h}\cdot \norm{x_n}}{\abs{ h^Tx_n}}}
\label{corr1}
\end{eqnarray}
\end{corollary}

Observe that if we were to turn this bound into a minimization objective, we would minimize an exponential loss and maximize the minimum margin. 

Denote by $\gamma_n^h=\frac{h^Tx_ny_n}{\norm{h}\cdot\norm{x_n}}$ the margin of the point $x_n$ with respect to the hyperplane defined by $h$. Using the Taylor expansion for $\exp(\cdot)$ we can write:
{\small
\begin{equation}
 \frac{1}{N}\sum_{n=1}^N \exp\(- \frac{h^Tx_ny_n}{\norm{h}\cdot\norm{x_n}}\) =
 \frac{1}{N}\sum_{n=1}^N \exp\(-\gamma_n^h\) =
1 - \frac{1}{N}\sum_{n=1}^N \gamma_n^h + \frac{1}{N}\sum_{n=1}^N (\gamma_n^h)^2 - ... 
\nonumber %\label{taylor}
\end{equation}
}
Now, observe that the minimizer of this term in our generalization bound implies that the average of 
the empirical margin distribution is maximized and its second moment (so also the variance)
is minimized. Hence, replacing the exponential term with its second order Taylor approximation in the bound of eq.(\ref{corr1}) we obtain an objective that recovers the recently proposed and successful method of Large Margin Distribution Machine (LDM) \cite{Zhang}. 

Indeed, the LDM \cite{Zhang} was formulated as a quadratic objective, implemented in an
efficient algorithm that maximizes the sample mean and minimizes the sample variance of the observed margin distribution, in addition to maximizing the minimum margin. Its original motivation was a boosting bound of \cite{doubt}, given as a function of the average and `some notion of' variance of the empirical margin distribution, derived by entirely different means. Via a completely different route, here we obtain a new explanation of LDM -- namely as implementing an approximate minimizer of the bound in eq. (\ref{corr1}) -- as an instance of the principle of capturing benign geometry and naturally occurring structures by means of random projections of the data. 

\subsubsection{Linear combination of base classifiers: Connecting two views}\label{sec:boost}
In this section we depart from the linear model, and 
consider a linearly weighted ensemble of binary valued base learners
from the class $B=\{b:\X\times \{-1,1\}\}$, with weights $\alpha=(\alpha_1,\alpha_2,...,\alpha_T)$:
\begin{eqnarray}
F_{ens}=\ny x \rightarrow \sum_{t=1}^T \alpha_tb_t(x): b_t\in B, \sum_{i=1}^T\abs{\alpha_i}\le 1\zr
\end{eqnarray}
The bounds derived so far can be adapted to this class simply by replacing the empirical Rademacher complexity term of the unit-norm linear function class with that of the ensemble, as detailed in Sections \ref{pf:boost1}-\ref{pf:boost2}.

By adapting our Theorem \ref{DataspaceB}, we get:
\begin{corollary}\label{boost1} 
Fix any $k(\le T)$ positive integer, and $\delta>0$.
With probability $1-\delta$ w.r.t. the training set of size $N$, uniformly for all $\alpha_t, \sum_{t=1}^T |\alpha_t|\le 1$ and all $b_t\in B, t=1,...,T$, 
{\small
\begin{equation} %\hspace{-0.75cm}
\pr_{x,y}[\sum_{t=1}^T \alpha_t b_t(x)y \le 0] %&=& %\pr_{x,y}[\alpha^T b(x)y \le 0]\\
\le \frac{1}{N}\sum_{n=1}^N \min\(1,2f_k(\theta_{b(x_n)y_n}^\alpha)\) + %c\sqrt{\frac{V(B)}{N}}
%2c\sqrt{\frac{2k}{\pi}}
c\sqrt{\frac{k\cdot V(B)}{N}} + 3\sqrt{\frac{\log(2/\delta)}{2N}}
\end{equation}
}
where $V$ denotes VC-dimension, and $c$ is an absolute constant.
\end{corollary}
If we regard $k$ as the inverse of a margin parameter, then Corollary \ref{boost1} is analogous to the empirical margin distribution bound on boosting \cite{Shapire}, derived by very different means, which gave rise to the classic margin-based explanation for the performance of boosting. 

Moreover, we also obtain another classical view of boosting, namely that of loss minimization \cite{ShapireExplaining}, from the same principle, by adapting our Theorem \ref{DataspaceB_kx}. While both views of boosting have coexisted for a long time, to our knowledge there has not previously been a single generic principle to connect them. 

Indeed, applying Theorem \ref{DataspaceB_kx} with the choice 
$k(h,b(x)y):= \frac{2\|b(x)\|_2}{|\cos(\theta_{b(x)y}^{\alpha})|}\cdot \frac{\norm{\alpha}_2}{\norm{\alpha}_1}$, where $b$ is the vector of binary predictions $(b_t)_{t=1,...,T}$
we get the following: 
\begin{corollary}\label{boost2}
With probability $1-\delta$ w.r.t. the training set of size $N$, uniformly for all $\alpha_t, \sum_{t=1}^T |\alpha_t|\le 1$ and all $b_t\in B, t=1,...,T$, 
\begin{eqnarray}
\pr_{x,y}[\alpha^Tb(x)y\le 0] &\le&\frac{1}{N}\sum_{n=1}^N
2\exp\(-\frac{\alpha^Tb(x_n)y_n}{\norm{\alpha}_1} \) + 
3\sqrt{\frac{\log(2/\delta)}{2N}} \nonumber\\
&+&\(c\sqrt{\frac{V(B)}{N}} + 3\sqrt{\frac{\log(2)}{2N}}\)
\sqrt{2T}\max_n
\sqrt{\frac{\norm{\alpha}_1}{\abs{ \alpha^Tb(x_n)}}}
\end{eqnarray}
\end{corollary}
The dependence on $T$ comes from $\|b(x)\|_2 = \sqrt{T}$ that enters in cosine-margins. 

The point here is that, looking at the r.h.s. of the bound of Corollary \ref{boost2} \emph{as a minimization objective}, we recognise the first term is the well-known exponential loss of adaboost, and  the last term contains the inverse of the minimum margin -- together these recover a regularized adaboost \cite{ShapireExplaining}. %Observe also that both margin-dependent terms decrease with the margin.
Thus, these two different views of boosting can now be understood as manifestations of the same principle. 

\subsection{How good is the theory? -- An empirical assessment}
%\subsection{How good is the theory?}
 A good theory should be capable of explaining essential characteristics of learning. 
 So far we have seen that our theory is able to metamorphose into successful existing algorithms, and to explain new connections between well established approaches. 
In this section we pursue an empirical assessment of its informativeness.

\subsubsection{Ability to compare classifiers}
We generated a synthetic data set of $N=280$ points with $140$ sampled from each of two standard Gaussians in $\R^{20}$ centred at $-0.5\cdot \Eins$ and $+0.5\cdot \Eins$ respectively, where $\Eins$ is the vector with all entries 1. This setting is designed so that the two classes are linearly separable on any single feature, but in order to find the separating hyperplane that generalizes most effectively to new points from the same distribution all of the features are required. 

We trained a sequence of $L_q$ regularised logistic regression models \cite{Lq,Ng} on this data, with $q\in\{0.1,0.2,...0.9,1,2\}$. 
%It is known that choices of $q\le 1$ induce a sparse weight vector, and the smaller the $q$ the stronger is the sparsity bias. The 
Here the case $q=2$ corresponds to ridge regularization, $q=1$ corresponds to a version of the Lasso, and smaller values of $q$ promote increasingly sparse weight vectors. Now, the sparser the model the more features will be ignored by the classifier, thus we are deliberately driving the process of variable selection towards choosing suboptimal decision boundaries as $q$ decreases. Overall we therefore expect to see higher generalization error for smaller $q$ on these data. We are interested to see if the bound in our Theorem \ref{DataspaceB} is able to predict this behaviour \emph{from the training data alone}, that is, as opposed to error evaluation on a held-out subset of observations.

Figure \ref{ModelSelection} shows the values of the bound obtained as we vary both $q$ and the regularisation weighting parameter $\lambda$, side by side with the mean percentage of test error estimated on a hold-out set of 120 points generated from the same model. All values are averages obtained by 
5 independent repetitions of the data generating process.
The values of the bound we see on the left hand side figure are obtained by plugging the estimated classifier weights $\hat{h}$ of the trained logistic regression classifiers into the uniform bound of Theorem \ref{DataspaceB}, with parameters with $\delta=0.05, k=10$. Other choices of $k$ produced qualitatively similar results. The bound has no access to the hold-out set, yet we see a remarkable agreement between the behaviour of the values of our bound (left hand plot) and that of the hold-out test error estimates (right hand plot). As suggested by intuition and predicted by our bound, in the right hand plot we see that the held-out error is relatively unaffected by $\lambda$ when $q=2$ while the error increases as $q$ becomes smaller or $\lambda$ becomes larger. This demonstrates the potential for our bound to score and compare the generalization ability of different classifiers, based on training data alone.

\begin{figure}
\begin{center}
\includegraphics[width=12.5cm, height=4.5cm]{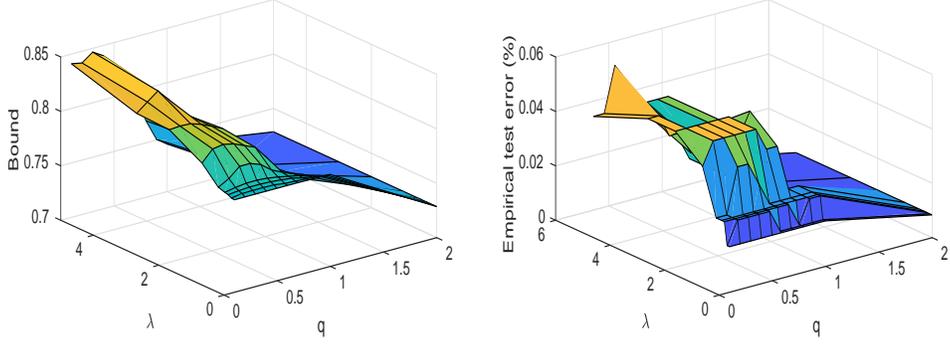}
\caption{\label{ModelSelection}
Agreement between the behaviour of the bound of Theorem \ref{DataspaceB} (on the left) and the hold-out test error estimate (on the right), as the hyper-parameters of the $L_q$-regularised classifier are varied. Although our bound still greatly overestimates the error, it has very similar behaviour to the test error estimate and is always non-trivial (less than 1).}
\end{center}
\end{figure}

\subsubsection{Ability to generate a new classifier}\label{algo}
To further assess the informativeness of the bound in Theorem \ref{DataspaceB}, 
we generate a new algorithm from it and test its generalization ability against the gold-standard SVM and related recent algorithms. A roughly similar idea was pursued in \cite{Garg03}, however their objective function was a heuristic simplification of the bound from \cite{Garg} as the latter was much too loose for direct use.   
With this aim in mind, and without any ambition of a computationally efficient approach, we turn our basic bound from Theorem \ref{DataspaceB} into a classifier by minimizing the bound.  

To this end, first we observe that so far it was sufficient for our purposes to consider $h$ that goes through the origin, however in this context this is worth fine-tuning. The reason is that adding an intercept in the usual manner (by padding the inputs with a dummy feature of ones) when working with the exact form of $f_k(\cdot)$, may not achieve the best achievable cosine values. 
Instead, we incorporate a new parameter vector $z$ that allows us to slightly shift the data simultaneously with learning $h$ to minimize the objective formed by the $h$-dependent terms of the bound. Thus, we minimize the following: 
\bqn
Obj(h,z) = \sum_{n=1}^N\frac{\Gamma(k)}{\(\Gamma(k/2)\)^2} \int_0^{\frac{1-a_n(h,z)}{1+a_n(h,z)}} 
\frac{v^{(k-2)/2}}{(1+v)^k} dv \label{obj}
\eqn
where 
\bqn
a_n(h,z):= \cos(\theta_{(\phix_n-z)y_n}^h)
= \frac{h^T}{\norm{h}}\frac{(\phix_n-z)y_n}{\norm{\phix_n-z}}
\eqn

We initialise $z$ in a low density region (e.g. the mid-point between data centres, $z_0$) and fine-tune it from the data in a small neighbourhood of $z_0$. Before doing so, let us theoretically show that $z$ is indeed learnable in this way. Replace $\phix$ by $\phix-z$ throughout, and define the modified
function class
\bqn {\mathcal F}_k^{shift}(h,z) =\ny \phix\rightarrow \frac{h^T}{\norm{h}}\cdot \frac{(\phix-z)}{\norm{\phix-z}} : 
h,z\in \R^d \zr
\eqn
%SHIFT LEMMA
\begin{lemma}\label{l:shift} 
Let ${\mathcal B}(z_0, r)$ be the Euclidean ball of radius $r$ centred at $z_0$.\\
Fix $\epsilon>0$, and suppose we can choose $r > 0$ small enough so that\\ 
%$\min_{n}\sup_{z\in {\mathcal B}(z_0,r)} \|x_n-z\|_2 >c >0$, and $r/c \le \epsilon/\sqrt{N}$.
$\sup_{z\in {\mathcal B}(z_0,r)} \frac{r}{\sqrt{N}}\sum_{n=1}^N \frac{1}{\|x_n-z\|}\le\epsilon$.
Then:
\bqn
\hat{\rad}({\mathcal F}_k^{shift}(h,z))\le \frac{1+\epsilon}{\sqrt{N}}
\eqn
\end{lemma}
Note the technical condition on Lemma \ref{l:shift} does not depend on the class labels, and only requires a low-density region around the centre of the input set.
%[The above can be revised to frame it as a kernel learning problem and point out the difference from it it since we optimize over a (small radius) continuous domain.]

%\subsubsection{Algorithm}
Now we are ready to turn Theorem \ref{DataspaceB} into an algorithm.
%From the computational point of view we should note that minimizing eq.(\ref{obj})
%simultaneously in $h$ and $c$ is a challenging non-linear optimization problem.
The gradients w.r.t. $h$ and $z$ are as follows:
\bqn
\Delta_h %&=& \ell_k'(a_n(h,z))\cdot \frac{\partial a_n(h,c)}{\partial h}\nonumber\\
&=& \frac{1}{N} \sum_{n=1}^N
\ell_k'(a_n(h,z))\cdot
\frac{(\phix_n-z)y_n}{\norm{\phix_n-z}}\cdot
\( I_d-\frac{hh^T}{\norm{h}^2}\)\frac{1}{\norm{h}}\nonumber\\
\Delta_z %&=&  \ell_k'(a(h,c))\cdot \frac{\partial a(h,c)}{\partial c}\nonumber\\
&=& - \frac{1}{N} \sum_{n=1}^N \ell_k'(a_n(h,z))\cdot
\frac{h y_n}{\norm{h}}\cdot \(I_d - \frac{(\phix_n-z)(\phix_n-z)^T}{\norm{\phix_n-z}^2}\)\cdot
\frac{1}{\norm{\phix_n-z}}\nonumber
\eqn
where $\ell_k'(a_n(h,z)) = 
-\frac{\Gamma(k)}{2^{k-1}(\Gamma(k/2))^2} (1-a_n(h,z)^2)^{\frac{k-2}{2}}$.
We used numerical integration to evaluate the objective, and a generic nonlinear optimizer\footnote{\url{http://learning.eng.cam.ac.uk/carl/code/minimize/}} 
that employs a combination of conjugate gradient and line search methods.

Figure \ref{illus} illustrates the classifier that results from optimizing our bound, in comparison with two alternatives: the gold-standard SVM -- which maximizes the minimal margin and disregards other geometry -- and the direct zero-one loss minimizer from \cite{01} -- which in our framework corresponds to choosing a very large value for $k$. It is most apparent that the new classifier 
is more robust against unessential detail in the data and captures the essential geometric structure. 

\begin{figure}
\centering
\includegraphics[width=7cm,height=6cm]{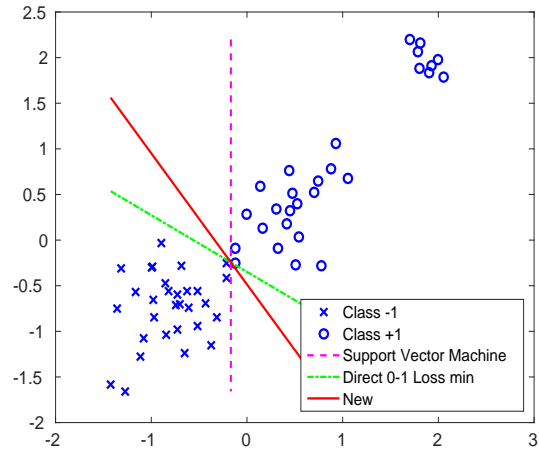}
\caption{\label{illus} Comparison of our bound-optimizing classifier (`New') with two popular alternatives. Note the very different separating planes implemented by each approach. In particular the orientation of the SVM boundary is exclusively influenced by the four support vectors.}
\end{figure}

\begin{table}[ht]
\centering
\caption{\label{T:results} Mean test error rates $\pm$ one standard error for our bound optimizer and some comparisons. Bold font indicates a significant improvement in test error over SVM at the 0.05 level using a paired t-test. Italic font in last two columns indicates performance was significantly worse than SVM for the competing methods on the corresponding dataset. On these data our approach was never significantly worse than SVM.}
 \begin{tabular}{l cc c c c  c  } 
 \hline
Data set &$N$ &$d$  & New                     & SVM              & Zero-one Loss        & LDM \\[0.5ex] \hline
Australian &690 & 14 &\textbf{0.137$\pm$ 0.015} & 0.148$\pm$ 0.013 & 0.156$\pm$0.077 & \emph{0.149$\pm$ 0.014}\\% our run
German & 1000 & 24  &\textbf{0.260$\pm$ 0.018} & 0.280$\pm$ 0.016 & 0.264$\pm$0.021 &  \emph{0.315$\pm$0.015} \\ %our run
Haberman &306 &3   &\textbf{0.265$\pm$ 0.025} & 0.285$\pm$ 0.050 & 0.268$\pm$0.024 & 0.276$\pm$0.030 \\ % our run 
Parkinsons &195 &22    &\textbf{0.141$\pm$ 0.032} & 0.221$\pm$ 0.049 & 0.141$\pm$ 0.036  &   0.135$\pm$0.034\\% cf paper 
PlRelax &182 & 12 &\textbf{0.285$\pm$ 0.029} & 0.361$\pm$ 0.166 & \emph{0.299$\pm$0.035} &0.290$\pm$0.051\\% cf paper
Sonar &208 &60 & 0.256$\pm$ 0.045          & 0.271$\pm$ 0.036 & 0.245$\pm$0.044 &  0.264$\pm$0.044\\% cf paper 
[1ex]
\hline
\end{tabular}
\end{table}

Experimental tests on UCI data sets \cite{UCI} are presented in Table \ref{T:results}. For each data set 
we performed 50 independent splits into two halves. Our parameter $k$, and SVM's C parameter,
were set by cross-validation on the training half. The error rates on the held-out testing half of the data are reported in Table \ref{T:results} in comparison with those of SVM (linear kernel). Despite the non-convex optimization involved in optimizing our bound, we observe improved generalization performance. %in 5 out of 6 data sets tested. 
These results support the view that our theory is \emph{good} in the sense that it is in line with practical experience and indeed seems to capture essential characteristics of learning that govern generalization.

%------------------------------------------
\section{Proofs of Main Results}\label{sec:Pfs}

\subsection{Flipping probability proofs}\label{sec:Pfs:flips}
\subsubsection{Proof of Lemma \ref{flips}}
Without loss of generality (w.l.o.g.) we can take $\|\x  \| = \|\y  \| = 1$.
%\begin{enumerate}

Rewrite the probability of interest as the following:
{\small
\bqn
\pr_R\ny \x^TR^TR\y \le 0\zr = 
\pr_R\ny \|R(\x+\y )\|^2 - \|R(\x-\y )\|^2 \le 0  \zr \label{flip_paral}
\eqn
}
Now, observe that the two terms $\|R(\x+\y )\|^2$ and $\|R(\x-\y )\|^2$ are statistically independent. Indeed, each component of the random vectors $R(\x+\y )$ and $R(\x-\y )$ are Gaussian distributed, and denoting by $R_i$ the $i$-th row of $R$, it is easy to verify that 
$\cov(R_i(\x+\y ),R_i(\x-\y ))=\|\x\|^2-\|\y \|^2 = 0$ (since w.l.o.g. we assumed $\|\x\|=\|\y \|=1$). Likewise, 
$\cov(R_i(\x+\y ),R_j(\x-\y ))_{i\neq j}=0$ by the independence of the rows of $R$.

The variances are $\var(R(\x+\y ))=\|\x+\y \|^2$ and $\var(R(\x-\y ))=\|\x-\y \|^2$. Hence, denoting 
\begin{eqnarray}
U^2 :=\left\|R\frac{\x+\y }{\|\x+\y \|}\right\|^2; \text{~~~~}
V^2 := \left\|R\frac{\x-\y }{\|\x-\y \|}\right\|^2
\end{eqnarray}
these  are independent standard $\chi^2$ variables. 
Therefore, we may rewrite eq. (\ref{flip_paral}) as the following:
\bqn
\pr_R\ny \x^TR^TR\y \le 0\zr 
&=&\pr_{U^2,V^2}\ny  U^2\|\x+\y \|^2 < V^2\|\x-\y \|^2   \zr \label{st} \nonumber\\
&=&\pr_{U^2,V^2}\ny \frac{ U^2}{V^2} < \frac{\|\x-\y \|^2}{\|\x+\y \|^2}  \zr
\eqn
where the fraction on the l.h.s. is $F$-distributed. Further, observe that the r.h.s. is
\bqn
\frac{\|\x-\y \|^2}{\|\x+\y \|^2} = \frac{2-2\x^T\y}{2+2\x^T\y}=\frac{1-\cos(\theta)}{1+\cos(\theta)}
\eqn
since $\|\x\|=\|\y \|=1$. Denoting this value by $\psi$,  
the integral of the cumulative density function of the F-distribution at $\psi$ gives us the result stated in eq. (\ref{eq:zform}).\\ 

%\item
The upper bound in eq. (\ref{ball}) is already known from convex geometry \cite{ball} if we notice the geometric interpretation of eq. (\ref{eq:zform})
as the area of a hyperspherical cap with angle $2\theta$ divided by the surface area of the corresponding sphere when $\cos(\theta)>0$ \cite{icml13}.
Instead, here we give a direct and elementary proof using the Chernoff bounding technique.

It is sufficient to consider the case $\cos(\theta)>0$, and the case $\cos(\theta)\le 0$ follows in the same way by symmetry.
When $\cos(\theta)>0$, %which implies $\x^T\y > 0$, 
we rewrite the r.h.s. of eq. (\ref{st}) as the following:
\bqn
\pr\{(R\x )^{T}R \y  \le 0\} 
&=& \pr_{U^2,V^2}\ny  -(\cos(\theta) +1)U^2 -(\cos(\theta)-1)V^2 > 0  \zr \nonumber\\
&\le& \E[\exp\ny -\lambda(\cos(\theta)+1)U^2 -\lambda (\cos(\theta)-1)V^2) \zr ] \nonumber \\%\label{before_indep} \\
&=& (1+2\lambda (\cos(\theta)+1))^{-k/2} (1+2\lambda(\cos(\theta)-1))^{-k/2} \label{l} \label{indep}
\eqn
for all $\lambda>0$ such that %$-\lambda(\cos(\theta)+1) \le 1/2$ and 
%$-\lambda(\cos(\theta)-1) \le 1/2$.
$1+2\lambda(\cos(\theta)-1) >0$.
In the last line we used that $U^2$ and $V^2$ are independent $\chi^2$ variables.

After straightforward algebra, the r.h.s. of eq.(\ref{l}) equals:
\bqn
 (1+4\lambda\cos(\theta) -4\lambda^2 \sin^2(\theta))^{-k/2} \nonumber
\eqn
Minimising this w.r.t. $\lambda$ gives that the optimal $\lambda$ satisfies $2\sin^2(\theta)\lambda=\cos(\theta)$.

So, if $\theta\neq 0$ then $\lambda=\frac{\cos(\theta)}{2\sin^2(\theta)}$ --
which satisfies the condition required above. In turn, if $\theta=0$ then the probability of interest is trivially 0 
so the upper-bound we derive holds in both cases.

Plugging back, after cancellations we get for the case $\x^T\y>0$ that:
\begin{eqnarray*}
\pr\{(R\x )^{T}R \y  \le 0\} 
&\le& \left( 1+\frac{\cos^2(\theta)}{\sin^2(\theta)} \right)^{-k/2}\\
%&=&\left( 1-\frac{\frac{\cos^2(\theta)}{\sin^2(\theta)}}{1+\frac{\cos^2(\theta)}{\sin^2(\theta)}} \right)^{k/2}\\
&=& (\sin^2(\theta))^{k/2}\\
&=&(1-\cos^2(\theta))^{k/2}\\
&\le & \exp\ny-\frac{k}{2}\cos^2(\theta)\zr \;\;\;\;\; \square
\end{eqnarray*}

%---------------------------
\subsubsection{Proof of Lemma \ref{thm:flip2}} 

We start by rewriting as in eq. (\ref{flip_paral}).
But now the two quadratic terms are not independent in general (albeit they are uncorrelated), since these are now non-Gaussian. 
%Specifically, the step from eq. (\ref{before_indep}) to (\ref{indep}) does not hold for subgaussian $R$. 
Hence, we purse a different strategy.

Rewrite eq. (\ref{flip_paral}) by inserting the following expression, which evaluates to zero:
\bqn
k\sigma^2 \|\x+\y \|^2(1-\cos(\theta))-k\sigma^2 \|\x-\y \|^2(1+\cos(\theta)) \label{insert}
\eqn
Indeed, it is easy to check that this expression equals $k\sigma^2 (4h^Tx - 4\cos(\theta))=0$ because $\|\x\|=\|\y \|=1$.

Inserting (\ref{insert}) into (\ref{flip_paral}) gives the following:
\bqn
&&\pr \ny -[\|R(\x+\y )\|^2-k\sigma^2\|\x+\y \|^2(1-\cos(\theta))] +...\right. \nonumber \\
          &&{\text{~~~}}\left. [\|R(\x-\y )\|^2-k\sigma^2\|\x-\y \|^2(1+\cos(\theta))] >0 |\cos(\theta)>0\zr \label{prev}
\eqn
Exponentiating both sides, and employing Markov inequality, for any $\lambda>0$ the r.h.s. of eq. (\ref{prev}) is upper bounded by:
\bqn	
&&\E\left[\exp\ny - \lambda \(\|R(\x+\y )\|^2-k\sigma^2\|\x+\y \|^2(1-\cos(\theta))\) +... \right.\right. \nonumber\\
   &&\left.\left. \lambda \(\|R(\x-\y )\|^2-k\sigma^2\|\x-\y \|^2(1+\cos(\theta))\)\zr | \cos(\theta)>0\right] \label{mgfs}
\eqn
Next, we introduce a convex combination which will serve us to exploit the convexity of the exponential function. For any $\alpha\in(0,1)$, 
eq. (\ref{mgfs}) equals:
{\small
\begin{eqnarray*}
&=&\E\left[\exp\ny - \alpha\frac{\lambda}{\alpha}\(\|R(\x+\y )\|^2-k\sigma^2\|\x+\y \|^2(1-\cos(\theta))\) +... \right.\right. \nonumber\\
   &\text{~}&\left.\left. (1-\alpha)\frac{\lambda}{1-\alpha} \(\|R(\x-\y )\|^2-k\sigma^2\|\x-\y \|^2(1+\cos(\theta))\)\zr | \cos(\theta)>0\right]\\
&\le& \alpha\E\left[\exp\ny -\frac{\lambda}{\alpha}[\|R(\x+\y )\|^2-k\sigma^2\|\x+\y \|^2(1-\cos(\theta))]\zr | \cos(\theta)>0\right] +...\nonumber\\
&&(1-\alpha) \E\left[ \exp\ny  \frac{\lambda}{1-\alpha} [\|R(\x-\y )\|^2-k\sigma^2\|\x-\y\|^2(1+\cos(\theta))]\zr | \cos(\theta)>0\right] \nonumber
\end{eqnarray*}
}
where we used Jensen's inequality in the last line.

Now, $\lambda_1:=\frac{\lambda}{\alpha}$ and $\lambda_2:= \frac{\lambda}{1-\alpha}$ are two free parameters each of 
which may be optimised independently because we can take $\lambda=1/(1/\lambda_1 +1/\lambda_2)$. Further,
since $R$ was subgaussian, we have two sub-exponential moment generating functions in eq. (\ref{mgfs}) that
are identical to those that appear in the proof of the two sides of the Johnson-Lindenstrauss lemma (JLL) in \cite{DasGup02}, 
but now with $\epsilon:=\cos(\theta)\in(0,1)$ playing the role of the distortion parameter. Hence, by the same arguments as in JLL, 
both expectations above are upper-bounded by $\exp(-k\epsilon^2/8)=\exp(-k\cos^2(\theta)/8)$. So we obtain the upper bound:
\bqn
 \alpha \exp(-k\cos^2(\theta)/8) +(1-\alpha) \exp(-k\cos^2(\theta)/8) \nonumber\\
=\exp(-k\cos^2(\theta)/8) \nonumber
\;\;\;\;\; \square
\eqn

%%%%%%%%%%%%%%%%%%%%%%%%%%%%%%%%%%%%%%%%%%%%%%%%%%%%%%%%%%%%%%%%%%%%%%%
\subsection{Proofs of bounds for compressive classifiers}\label{pf:compr}
%%%%%%%%%%%%%%%%%%%%%%%%%%%%%%%%%%%%%%%%%%%%%%%%%%%%%%%%%%%%%%%%%
\subsubsection{Proof of Theorem \ref{thm:riskb2}}  
For a fixed instance of $R$, a classical VC bound \cite{BarMen02} gives that $\forall \delta\in(0,1)$ w.p. $1-\delta$ over the random draws of the training set, the following holds uniformly for all $h_R\in R^k$:
{\small
\begin{equation}
\pr_{x,y}[(\erm_R^T Rx+b)y\le 0]\le  \frac{1}{N}\sum_{n=1}^N 
\Eins((\erm_R^T Rx_n+b)y_n\le 0) + c\sqrt{\frac{k+1+\log(1/\delta)}{N}}
\label{eq:vc}
\end{equation}
}
for some absolute constant $c>0$. 
Since $(\erm_R,b)$ is the ERM classifier in the RP-ed space, its empirical error is bounded by that of the homogeneous ERM classifier on $k$-dimensional inputs (denoted $\erm'_R$), which in turn is 
bounded for any choice of $h\in \hypclass$, %the empirical error term in eq. (\ref{eq:vc}) is upper bounded 
as the following:  $\frac{1}{N}\sum_{n=1}^N  \Eins((\erm_R^T Rx_n+b)y_n\le 0)$
{\small
\begin{eqnarray*}
...&\le & \frac{1}{N}\sum_{n=1}^N  \Eins(\erm_R^{'T} Rx_n y_n\le 0) \le
\frac{1}{N}\sum_{n=1}^N \Eins(h^TR^T Rx_ny_n\le 0)\\
&=& \frac{1}{N}\sum_{n=1}^N[ \Eins\(h^TR^TRx_ny_n\le 0\) - 
\Eins\(h^T x_ny_n \le 0\) ] + \Eins\(h^T x_ny_n\le 0\)\\
&=& \frac{1}{N}\sum_{n=1}^N[
\Eins\(h^TR^TRx_ny_n\le 0\)\Eins\(h^Tx_ny_n > 0 \) \\ 
&-& \Eins\(h^TR^TRx_ny_n > 0\) \Eins\(h^Tx_ny_n \le 0 \)
] + \frac{1}{N}\sum_{n=1}^N\Eins\(h^Tx_ny_n\le 0\)\nonumber \\
&\le & 
\frac{1}{N}\sum_{n=1}^N
\Eins\(h^TR^TRx_ny_n\le 0\) \Eins\(h^Tx_ny_n > 0 \)  + \frac{1}{N}\sum_{n=1}^N\Eins\(h^Tx_ny_n \le 0\) 
\end{eqnarray*}
}
Plugging this back we have $\forall h\in \hypclass$,
\begin{eqnarray*}
\pr_{x,y} \{(\erm_R^T Rx+b)y \le 0 \} &\le&
\underbrace{
\frac{1}{N}\sum_{n=1}^N \Eins(h^T R^TRx_ny_n \le 0) \cdot  \Eins(h^Tx_ny_n > 0)
}_{=:T}\\ &+&
\frac{1}{N}\sum_{n=1}^N \Eins(h^T x_ny_n \le 0) +c\sqrt{\frac{k+1+\log(1/\delta)}{N}}
\end{eqnarray*}
Now, we bound the term $T$ from its expectation.
Since $T\in [0,1]$, by the H\"offding bound it holds $\forall\epsilon > 0$,
\begin{eqnarray}
\pr\{ T \geq \E_R[T] +\epsilon \} \leq \exp\left(-2\epsilon^2\right) 
\end{eqnarray}
which implies:
\begin{eqnarray}
T \leq \E_R[T] +\sqrt{\frac{1}{2}\log\frac{1}{\delta}}
\end{eqnarray}
We combine this with a Markov inequality, which is tighter for small $\E_R[T]$, while the H\"offding bound is tighter for small values of $\delta$.
Taking the minimum of these two bounds completes the proof of the first part.

To obtain the excess risk bound we use a standard result for ERM classifiers, namely that eq. (\ref{eq:vc}) implies the following:
{\small
\begin{eqnarray*}
\pr_{x,y} [(\hhr^T Rx+b)y \le 0 ]
\le \pr_{x,y} [(h^{*T}_R Rx+b)y \le 0] +2c\sqrt{\frac{k+1 +\log(1/\delta)}{N}} \label{eq:VC2}
\end{eqnarray*}}
Since $h^*_R$ is the optimal classifier in $\hypclass_R$, and $Rh^* \in \hypclass_R$, we can write: %[SEEMS TO ME WE CAN HAVE ANY $h$ HERE AS PREVIOUSLY, not only $h^*$ in the following? - TO CHECK]
{\small
\begin{eqnarray*}
\pr_{x,y}[(h^{*T}_R Rx+b) y \le 0] &\le&  \pr_{x,y}[h^{*T}R^TRxy \le 0]%:= S^*  
\\
&\le &\underbrace{\E_{x,y}[ \Eins(h^{*T}R^TRxy\le 0)\cdot \Eins(h^{*T}xy > 0)]}_{T^*}
+ \pr_{x,y}[h^{*T}xy \le 0] 
\end{eqnarray*}
}
and $T^*$ is bounded using the combination of H\"offding and Markov bounds employed before.
$\square$

\subsubsection{Proof of Theorem \ref{riskb}}
Again we start from the standard VC bound in $\R^k$, eq. (\ref{eq:vc}), 
and bound the empirical error as:
\begin{equation}
\frac{1}{N}\sum_{n=1}^N \Eins\{(\erm_R^TRx_n+b)y_n\le 0\} \le \frac{1}{N}\sum_{n=1}^N \Eins\{h^TR^TRx_n y_n \le 0\} 
:= S %\label{eq:S}
\end{equation}
which is valid for any $h\in \hypclass$.

When $R$ is Gaussian, we have the exact form of $\E_R[S]=\frac{1}{N}\sum_{n=1}^N f_k(\theta_{x_ny_n}^h)$, hence we directly bound $S\in [0,1]$, 
using the same combination of H\"offding and Markov inequalities previously. 

The excess risk bound in the second part of the Theorem statement %in Eq. (\ref{eq:bound2G})
follows as a straightforward consequence as before. $\square$

%-------------------------------------------
\subsubsection{Proof of Theorem \ref{thm:uf}}
We need the following Lemma.

\begin{lemma}\label{uf}(Uniform bound on sign flipping)
Fix $h\in\R^d, \norm{h}=1$ w.l.o.g.
Let $R$ be an isotropic subgaussian random matrix with independent rows having subgaussian norm bounded as $\norm{R_i}_{\psi_2}\le K$. 
Let $T_{h,\gamma}^+$ be as defined in eq.(\ref{eq:set}).
Then, for any fixed $h$, $\forall \delta\in(0,1)$, w.p. $1-\delta$ w.r.t. draws of $R$, 
\begin{eqnarray}
P_R\ny \exists u\in T_{h,\gamma}^+: h^TR^TRu \le 0\zr < \delta 
\end{eqnarray}
provided that $k \ge 
CK^4\(w(T_{h,\gamma}^+) + \sqrt{\log(2/\delta)}\)^2 / \gamma
%+2CK^2\(w(T)+\sqrt{\log(1/\delta)})\)\sqrt{k}
$
for some absolute constant $C$.
\end{lemma}

\begin{proof}[Proof of Lemma \ref{uf}]
By the parallelogram law,
{\small
\begin{eqnarray*}
-\frac{h^TR^TRu}{k} &=& \frac{1}{4k}\( \norm{R(h-u)}^2 - \norm{R(h+u)}^2 \)\\
 &=&  \frac{1}{4}\(
 \frac{\norm{R(h-u)}^2}{k} - \norm{h-u}^2 \) - 
 \frac{1}{4}\(
\frac{\norm{R(h+u)}^2}{k} - \norm{h+u}^2
 \) 
 -h^T u
\end{eqnarray*}
}
Now each of the brackets is an empirical process, and we can make use a the following result (Theorem 1.4 from \cite{Liaw}) that bounds the suprema of such processes. %\url{http://www-personal.umich.edu/~romanv/papers/lmpv-deviation.pdf}), which improves on seminal work by Klartag \& Mendelson \cite{KM} that first linked together RP with empirical processes.

\textbf{Theorem}[Liaw et al. \cite{Liaw}]. 
Let $R$ be an isotropic subgaussian random matrix with independent rows having subgaussian norm bounded as $\norm{R_i}_{\psi_2}\le K$. 
Let $T$ be a bounded subset of $R^d$, and denote its radius by $\text{rad}(T)=\sup_{u\in T}\norm{u}$. 
Then there is an absolute constant $C$ s.t. with probability $1-\delta$ w.r.t. the random draws of $R$, 
{\small
\begin{eqnarray}
\sup_{u\in T} \left| \frac{\norm{Ru}_2^2}{k} - \norm{u}_2^2 \right|
&\le& \frac{1}{k} \left[
 C^2K^4\(w(T) + \text{rad(T)}\sqrt{\log(1/\delta)}\)^2 \right. \nonumber\\
 &+& \left. 2CK^2 \text{rad}(T)\(w(T)+ \text{rad}(T) \sqrt{\log(1/\delta)} \)
 \sqrt{k}\right]\label{V}
\end{eqnarray}
}
%Alternatively we could use Dirksen \cite{Dirksen}, then the r.h.s. would be $\lesssim$
%\begin{eqnarray}
%K^4 w(T)^2+\sqrt{k}K^2\(\text{rad}(T)w(T)+\text{rad}(T)^2\sqrt{\log(1/\delta)}\)+K^2\text{rad}(T)^2\log(1/\delta) 
%\end{eqnarray}
 
Denoting by $V(T',\delta)/k$ the r.h.s. of eq.(\ref{V}), where $T'$
is the set whose Gaussian width appears in it,
we get w.p. $1-2\delta$:
\begin{eqnarray}
\sup_{u\in T_{h,\gamma}^+} [- \frac{h^TR^TRu}{k}] \le \frac{V(T_{1,h,\gamma}^+,\delta)}{4k}+
\frac{V(T_{2,h,\gamma}^+,\delta)}{4k}-\gamma \label{putthisto0}
\end{eqnarray}
where 
\begin{eqnarray}
T_1 = \{ h-u : u\in T_{h,\gamma}^+ \}; \;\;\;\;\; T_2 = \{ h+u : u\in T_{h,\gamma}^+ \}
\end{eqnarray}
Now, since $h$ is a fixed vector, and the Gaussian width is invariant to translation, $w(T_1)=w(T_2)=w(T_{h,\gamma}^+)$.

We need to require that the r.h.s. of eq. (\ref{putthisto0}) is non-positive.
Since $\|h\|=\|u\|=1$, and   
$h^Tu=\cos(\theta_u^h)\ge \gamma$,  this is equivalent to requiring that:
\begin{eqnarray}
\frac{V(T_{h,\gamma}^+,\delta)}{2k}\le \gamma
\end{eqnarray}
Hence, for 
\begin{eqnarray}
k\ge \frac{V(T_{h,\gamma}^+,2\delta)}{2\gamma} =
\frac{\Omega(w(T_{h,\gamma}^+)^2)}{\gamma}
\end{eqnarray}
we have the statement of Lemma \ref{uf},
noting again that by definition $T_{h,\gamma}^+$
is a set of unit vectors, so its radius is 1.
\end{proof}

\begin{proof}[Proof of Theorem \ref{thm:uf}].
Again, for a fixed instance of $R$, we have the uniform VC bound of eq. (\ref{eq:vc}).
%W.l.o.g. take $\norm{h}=1, \norm{x}=1$. % doesn't seem to be needed at this point.
We upper bound the empirical error of the ERM classifier as the following:
$\frac{1}{N}\sum_{n=1}^N  1((\erm_R^T Rx_n+b)y_n\le 0) \le ...$
{\small
\begin{eqnarray*}
...&\le & \frac{1}{N}\sum_{n=1}^N \Eins(h^TR^T Rx_ny_n\le 0)\nonumber\\
&=& \frac{1}{N}\sum_{n=1}^N[\Eins\(h^TR^TRx_ny_n\le 0\) - 
\Eins\(\cos(\theta_{x_ny_n}^h) \le \gamma\) ] + \Eins\(\cos(\theta_{x_ny_n}^h)\le \gamma\)\nonumber \\
&=& \frac{1}{N}\sum_{n=1}^N[
\Eins\(h^TR^TRx_ny_n\le 0\)\Eins\(\cos(\theta_{x_ny_n}^h) > \gamma \)\nonumber \\ &-&
\Eins\(h^TR^TRx_ny_n > 0\)\Eins\(\cos(\theta_{x_ny_n}^h) \le\gamma \)
] + \Eins\(\cos(\theta_{x_ny_n}^h)\le \gamma\)\nonumber \\
&\le & 
\frac{1}{N}\sum_{n=1}^N
\Eins\(h^TR^TRx_ny_n\le 0\)\Eins\(\cos(\theta_{x_ny_n}^h) > \gamma \)  + \Eins\(\cos(\theta_{x_ny_n}^h)\le \gamma\)\label{firstterm0}
\end{eqnarray*}
}
Now, by Lemma \ref{uf}, for the stated $k$ the the first term on the r.h.s. is zero w.p. $1-\delta$.
This completes the proof. 
\end{proof}

%%%%%%%%%%%%%%%%%%%%%%%%%%%%%%%%%%%%%%%%%%%%%%%%%%%%%%%%%%%%
\subsection{Proofs for dataspace bounds}\label{pf:dataspace}
\subsubsection{Proof of Theorem \ref{DataspaceB}}
%We start by decomposing the 0-1 loss into a distortion term, defined as the absolute difference between the value of the loss before and after random projection on average, and the error on the compressed problem:
%\begin{equation}
%\pr_{x,y}[h^Txy\le 0] \le\pr_{x,y}\abs{\Eins(h^Txy\le 0)-\E_R[h^TR^TRxy\le 0]} + \pr_{x,y,R}[h^TR^TRxy\le 0]
%\end{equation}
%In the above, $\Eins(\cdot)$ returns $1$ if its argument is true and $0$ otherwise.
%
%The distortion term essentially measures the compressibility of the problem, i.e. to what extent the data geometry allows the high dimensional problem to be solvable in lower dimension. We can write this as the following:
%\begin{eqnarray*}
%&&\E_{x,y}\abs{\Eins(h^Txy\le 0)-\E_R(\Eins(h^TR^TRxy\le 0))}\\
%&&\;\;\;\; = \E_{x,y}\ny (1 - f_k(\theta_{\phix y}^h)) \cdot 1(h^T\phix y \le 0) +f_k(\theta_{\phix y}^h)\cdot %1(h^T\phix y> 0)\zr
%\end{eqnarray*}
%
%Plugging this back, the inequality
The following inequality is immediate:
\begin{eqnarray*}
\pr_{x,y}[h^Txy\le 0]&\le&\E_{x,y}[\Eins(h^T\phix y \le 0)+2f_k(\theta_{\phix y}^h) \Eins(h^T\phix y > 0)]\label{line2}\\
&=&\E_{x,y}[\min(1,2f_k(\theta_{xy}^h))]
\end{eqnarray*}
and will turn out to provide us a Lipschitz function of $\cos(\theta_{xy}^h)$.

The classical Rademacher complexity based risk bound \cite{KP02} (see also Theorem 3.1. in \cite{Mohri}) yields
for any fixed $k$ positive integer the following:
\bqn
\pr_{x,y}[h^T\phix y \le 0] 
&\le& \frac{1}{N}\sum_{n=1}^N %\ny 1(h^T\phix_n y_n \le 0) +2f_k(\theta_{\phix_n y_n}^h)\cdot 1(h^T\phix_n y_n > 0)\zr
\min(1,2f_k(\theta_{x_ny_n}^h))
\nonumber \\
&+&2\hat{\rad}_N(G_k)+3\sqrt{\frac{\log(2/\delta)}{2N}} 
\label{RadBound0}
\eqn
where $\hat{\rad}_N(\cdot)$ denotes the empirical Rademacher complexity of the function class in its argument, and 
we defined the function class $G_k$: 
\begin{equation}
G_k =\{ u \rightarrow 
%\Eins(h^T u \le 0) +2f_k(\theta_{u}^h)\cdot \Eins(h^Tu > 0)
\min(1,2f_k(\theta_{u}^h))
: h\in \R^d \}
\end{equation}

To compute the Rademacher complexity, we rewrite this function class as a composition:
\bqn
G_k = \ell_k  \circ {\mathcal F}
\eqn
where
\bqn
\ell_k: [-1,1] \rightarrow [0,1],\;\;\;  \ell_k(a) &=& 
%2\frac{\Gamma(k)}{\(\Gamma(k/2)\)^2} \int_0^{\frac{1-a}{1+a}} \frac{z^{(k-2)/2}}{(1+z)^k} dz\cdot 1(a>0) + 1(a\le 0)
\min\(1,\;\; 2\frac{\Gamma(k)}{\(\Gamma(k/2)\)^2} \int_0^{\frac{1-a}{1+a}}\frac{z^{(k-2)/2}}{(1+z)^k} dz
\)
\nonumber\\
{\mathcal F} &=& \ny
u \rightarrow \frac{h^T}{\norm{h}} \frac{u}{\norm{u}} : h\in \R^d 
\zr\nonumber
\eqn

Now, $\ell_k$ is Lipschitz continuous with constant $L$ as follows. Since $\ell_k$ is constant on $a\in [-1,0]$, it is enough to check the Lipschitz property
on $a\in [0,1]$. By the Leibniz integration rule  we have that:
\bqn
\abs{\ell_k'(a)} &=& \left|
-2\frac{\Gamma(k)}{2^{k-1}(\Gamma(k/2))^2} (1-a^2)^{\frac{k-2}{2}}\right|\\
&\le & 2\frac{\Gamma(k)}{\(\Gamma(k/2)\)^2 2^{k-1}} = L
\eqn

We can further simplify the expression of $L$ by rewriting it into a Gamma function ratio. Using the duplication formula (\cite{a_and_s}, 6.1.18, pg
256): $\Gamma(2z) = (2\pi)^{-\half} 2^{2z - \frac{1}{2}} \Gamma(z)
\Gamma((2z+1)/2)$  with $z=k/2$, the expression of $L$
is equal to:
\begin{equation}
2\frac{2^{k-\half}\Gamma(k/2)\Gamma((k+1)/2)}{\sqrt{2\pi}
2^{k-1}(\Gamma(k/2))^{2}}
%\end{equation}
%\begin{equation}
= 2\frac{\Gamma(k/2)\Gamma((k+1)/2)}{\sqrt{\pi}\ (\Gamma(k/2))^{2}}
%\end{equation}
%\begin{equation}
= 2\frac{\Gamma((k+1)/2)}{\sqrt{\pi}\ \Gamma(k/2)}
\end{equation}
Now we use an inequality for Gamma function ratios, shown 
in \cite{Wendel}: $x(x+y)^{y-1} \le\frac{\Gamma(x+y)}{\Gamma(x)}\le x^y, \forall y\in[0,1]$, 
which yields:
\bqn
L&=& 2\frac{\Gamma(k/2+1/2)}{\sqrt{\pi}\;\;\Gamma(k/2)}
\le  \sqrt{\frac{2k}{\pi}}
\eqn

In consequence, by Talagrand's contraction lemma %(Lemma 4.2 in \cite{Mohri}), 
(e.g. Theorem 7 in \cite{MeirZhang})) we have:
\bqn
\hat{\rad}_N(G_k) &\le&  \sqrt{\frac{2k}{\pi}} \cdot \hat{\rad}_N({\mathcal F})\label{contr}
\eqn
Finally, since ${\mathcal F}$ is a linear function class in $h/\|h\|$, and both $h/\norm{h}$ and $\phix y/\norm{\phix}$ have unit norm, so by Theorem 4.3 in \cite{Mohri} we have
%\bqn
$\hat{\rad}_N({\mathcal F}) \le \frac{1}{\sqrt{N}}$
%\label{linRad}
%\eqn
Combining this with eqs (\ref{RadBound0}) and (\ref{contr}) %, and (\ref{linRad}) 
completes the proof.
$\square$

\subsubsection{Proof of Theorem \ref{DataspaceB_kx}} 
We use Theorem \ref{DataspaceB} with SRM on $k$. This allows us to select the value for $k$ after seeing the sample. Let this value be $k_{\max}=\max_{n=1}^N k(x_ny_n,h)$. Hence it holds w.p. $1-\delta$, $\forall h\in \H$ that:
{\small
\begin{eqnarray*}
\pr_{x,y}[h^T\phix y \le 0]\le 
\frac{1}{N}\sum_{n=1}^N \min(1,2f_{k_{\max}}(\theta_{\phix_n y_n}^h))+
\frac{2\sqrt{2}}{\sqrt{\pi}}\sqrt{\frac{k_{\max}}{N}}\\+
3\sqrt{\frac{\log(2/\delta)}{2N}} + 3\sqrt{\frac{\log(2)}{2}}\sqrt{\frac{k_{\max}}{N}} \label{fst}
\end{eqnarray*}
}
Finally, we use the fact that for any $u,h$, the function $\min(1,2f_k(\theta_{u}^h))$ is non-increasing with $k$, cf. Sec. 6.5 in \cite{BobThesis}. % or copy-paste it here? (carefully since notations slightly differ)
Therefore the first term on the r.h.s. in Theorem \ref{DataspaceB_kx} is a deterministic upper bound on the first term of eq.\ref{fst}, and all the other terms are identical. 
$\square$

\subsubsection{Proof of of Corollary \ref{boost1}}\label{pf:boost1} 
We apply Theorem \ref{DataspaceB} to the linear-convex aggregation in $F_{ens}$, but since here the inputs into this aggregation are the outputs of the base classifiers learned from the data, we need to replace the empirical Rademacher complexity contained in that bound with the following:
\begin{eqnarray}
\hat{\rad}_N(F_{ens}) &=& \frac{1}{N}\E_{\sigma}\sup_{\alpha,b}\sum_{n=1}^N \sigma_n\frac{\alpha^T b(x_n)}{\|\alpha\| \cdot \|b(x_n)\|} \nonumber\\
%&=&\frac{1}{N}\E_{\sigma}\sup_{\alpha,b}\sum_{n=1}^N \sigma_n\frac{\alpha^T}{\|\alpha\|_1}\cdot \frac{\|\alpha\|_1}{\|\alpha\|_2}\cdot \frac{1}{\|b(x_n)\|_2}b(x_n) \\
&=& \sup_{\alpha}\frac{\|\alpha\|_1}{\|\alpha\|_2} \frac{1}{N}\E_{\sigma}\sup_{\alpha,b}
\sum_{n=1}^N \sigma_n\frac{\alpha^T}{\|\alpha\|_1}\cdot 
\frac{1}{\|b(x_n)\|_2}b(x_n) \label{absconv}
\end{eqnarray}
Since $b(x_n)\in \{-1,1\}^T$, we have $\norm{b(x_n)}_2=\sqrt{T}, \forall x_n$. We also have $\frac{\|\alpha\|_1}{\|\alpha\|_2} \le \sqrt{T}$. Therefore, the elements of $F_{ens}$ belong to the absolute convex hull of $B$, so by Theorem 3.3. in \cite{Boucheron} we have the r.h.s. of eq. (\ref{absconv}) upper bounded by:
\begin{eqnarray}
\frac{1}{N}\E_{\sigma}\sup_{b}\sum_{n=1}^N \sigma_n b(x_n)
= \hat{\rad}_N(B) \le c\sqrt{\frac{V(B)}{N}}
\end{eqnarray}
for some absolute constant $c$.
The last inequality is a known link between the Rademacher complexities and VC dimension \cite{KP02}.
$\square$

%----------------------------------------------

\subsubsection{Proof of Corollary \ref{boost2}}\label{pf:boost2}
We apply Theorem \ref{DataspaceB_kx}, replacing the empirical Rademacher complexity, 
and we note that: 
{\small
\begin{eqnarray}
\sqrt{k(x_ny_n,h)} &=& \sqrt{\frac{2 \|b(x_n)\|}{|\cos(\theta_{b(x_n)y_n}^{\alpha})|}\cdot \frac{\norm{\alpha}_2}{\norm{\alpha}_1}} \le \sqrt{\frac{2\norm{\alpha}_1}{\abs{ \alpha^Tb(x_n)}}}\cdot \sqrt{T}
\label{kn}
\end{eqnarray}
}
since $\frac{\|\alpha\|_2}{\|\alpha\|_1} \le 1$, and $\|b(x_n)\|_2=\sqrt{T}$.
Multiplying together eqs. (\ref{kn}) and (\ref{absconv}) completes the proof.
$\square$

%----------------------------------------

\subsubsection{Proof of Lemma \ref{l:shift}} 
First, observe that for any fixed $z_0$, the empirical Rademacher complexity remains unchanged.
\bqn
\hat{\rad}({\mathcal F}_k^{shift})(h,z=z_0)=\E_{\sigma}\sup_{h\in \R^d}
\frac{1}{N} \left[\sum_{n=1}^N \sigma_n
\frac{h^T}{\norm{h}}\cdot \frac{(\phix_n-z_0)}{\norm{\phix_n-z_0}}\right]
= \frac{1}{\sqrt{N}}\nonumber
\eqn
Now, for any tolerance $\epsilon>0$, by the mean value theorem and Cauchy-Schwartz, we have:
\begin{eqnarray}
\hat{\rad}({\mathcal F}_k^{shift}(h,z))&=&\E_{\sigma}\sup_{h\in \R^d, z\in {\mathcal B}}
\frac{1}{N} \left[\sum_{n=1}^N \sigma_n
\frac{h^T}{\norm{h}}\cdot \frac{(\phix_n-z)}{\norm{\phix_n-z}}\right] \nonumber \\
&\le&\E_{\sigma}\sup_{h,z}\|\Delta_z\|\cdot \|z-z_0\| + 
\hat{\rad}({\mathcal F}_k^{shift}(h,z_0))\label{rad:shift}
\end{eqnarray}
where $\Delta_z$ is the gradient w.r.t. $z$ of the function in the argument of the supremum, so:
{\small
\begin{eqnarray}
 \| \Delta_z\| &=& \norm{ \frac{h^T}{\|h\|}\cdot \frac{1}{N}\sum_{n=1}^N 
\( I_d -\frac{(x_n-z)(x_n-z)^T}{\|x_n-z\|^2}\) \cdot \frac{\sigma_n}{\|x_n-z\|}}\\ 
&\le& \sup_{z\in {\mathcal B}(z_0,r)}\frac{1}{N} \sum_{n=1}^N \frac{1}{\|x_n-z\|}\label{eq:lmax}\\ 
%&\le& \frac{1}{\if_{z\in {\mathcal B}(z_0,r)}\min_{n=1}^N \|x_n-z\|} \le \frac{1}{c}
&\le& \frac{\epsilon}{r\sqrt{N}}
\end{eqnarray}
}
The inequality in eq. (\ref{eq:lmax}) is because 
$\lmax\( I_d -\frac{(x_n-z)(x_n-z)^T}{\|x_n-z\|^2} \)=1$. 
Plugging this back into eq.(\ref{rad:shift}) and 
noting that $\|z-z_0\|\le r$ gives:
\bqn
\hat{\rad}({\mathcal F}_k^{shift}(h,z)) \le
\frac{\epsilon}{\sqrt{N}}
+\hat{\rad}({\mathcal F}_k^{shift}(h,z_0))
= \frac{1+\epsilon}{\sqrt{N}}   \;\;\;\;\;\;\;\square \nonumber
\eqn

%%%%%%%%%%%%%%%%%%%%%%%%%%%%%%%%%%%%%%%%%%%%%%%%%%%%%%%%%%%

\section{Conclusions}

%The conclusion should draw all the work together and reiterate the main findings.
We proved novel risk bounds for binary classification with thresholded linear models in high-dimensional settings, which remain informative when the sample size is allowed to be smaller than the dimensionality of the training set of observations.
Throughout this work we mainly focused on linear classification though for nonlinear problems one could, in principle, replace $x$ with its feature space representation $\phi(x)$ induced by a fixed choice of kernel s.t. $K(x_1,x_2)=\phi(x_1)^T\phi(x_2)$, and the bounds we presented for the linear models here would then hold in unchanged form for the kernel-based function class. Of course various quantities will change under the mapping $\phi(\cdot)$ relative to the original data so such an approach sacrifices some practical advantages of our bounds -- addressing this remains for future work.

Our guarantees for learning the ERM classifier from randomly projected data improve and generalize earlier results, and highlight the role that data geometry plays in the success of the compressive classifier, and the ability of random projection to exploit this geometry.\\
Based on this insight, we gave new uniform bounds for zero-one loss classification in the full high-dimensional space that are tighter than VC bounds whenever the problem is compressible in the sense of possessing a low flipping probability.
In this context, the flipping probability can be viewed as a new and natural smooth loss that allows us to directly minimize the zero-one loss, despite this latter loss being discontinuous. Moreover it gains us access to the effects of small perturbations on low complexity sets despite the scale-insensitive zero-one loss, and without any regularization contraints. We used our results to draw connections between some existing successful classification approaches, including two different explanations of boosting.\\
We demonstrated the practical informativeness of our bounds empirically both by a simulation experiment and by constructing a well-performing classifier that minimizes our bound as the learning objective.\\
Future work will be required to extend these principles to other learning settings, to explore further algorithms that could be derived from the same principles, and to improve their computational efficiency.

%In the second part of the paper,we have put forward the idea that random projection can be used to capture benign geometry for learning, and as such it can provide a better and more informative theory to explain learning and generalization for tasks whose loss function of interest is scale-insensitive. In this work we developed this idea for binary classification, and gave new generalization bounds that remain informative in low sample size conditions without the use of an artificially imposed surrogate loss, and without any pre-defined restrictions on the norms of parameters or on the data domain of the input data. We tested our theory by its ability to explain existing successful learning methods as well as its ability to generate a new one. 

\subsection*{Acknowledgements} AK acknowledges funding from EPSRC Fellowship EP/P004245/1 ``FORGING: Fortuitous Geometries and Compressive Learning". Part of this paper was written while RJD was visiting the University of Birmingham on sabbatical -- RJD acknowledges financial support from the University of Waikato that made this visit possible. Both authors thank Dehua Xu for early work on the implementation of the algorithm in Section \ref{algo}.

%%%%%%%%%%%%%%%%%%%%%%%%%%%%%%%%%%%%

\end{document}